\documentclass[american,parskip=half,12pt,
]{scrartcl}

\usepackage{amssymb,amsmath,mathtools}
\usepackage[hyperref,amsmath,thmmarks]{ntheorem} 
\usepackage{aliascnt}

\usepackage{graphicx}
\usepackage{xcolor}
\usepackage{pgfplots}
\pgfplotsset{compat=1.18}
\usepgfplotslibrary{external}
\usepackage{float}
\mathtoolsset{showonlyrefs}
\usepackage{subcaption}

\usepackage[backend=biber, style=numeric-comp, maxbibnames=14, giveninits=true, backref=false
]{biblatex}

\usepackage[ruled,vlined]{algorithm2e}

\newrobustcmd{\MakeTitleCase}[1]{%
  \ifthenelse{\ifcurrentfield{title}}%
  {\MakeSentenceCase{#1}}%
  {#1}}
\DeclareFieldFormat[article,inbook,incollection,inproceedings]{titlecase}{\MakeTitleCase{#1}}
\AtEveryBibitem{\clearfield{issn}}
\AtEveryBibitem{\clearfield{month}}
\renewbibmacro{in:}{%
  \ifentrytype{article}{}{\printtext{\bibstring{in}\printunit{\intitlepunct}}}}

\usepackage[english,pdfusetitle,colorlinks=true,linkcolor=blue,citecolor=green!50!black,urlcolor=blue,bookmarks=true,
  pdfauthor={Michael Quellmalz},
]{hyperref}

\usepackage{mathabx}

\theoremseparator{.}
\newtheorem{lemma}{Lemma}

\newaliascnt{proposition}{lemma}

\aliascntresetthe{proposition}

\newaliascnt{corollary}{lemma}

\aliascntresetthe{corollary}

\newaliascnt{theorem}{lemma}
\newtheorem{theorem}[theorem]{Theorem}
\aliascntresetthe{theorem}

\theorembodyfont{\normalfont}
\newaliascnt{definition}{lemma}

\aliascntresetthe{definition}

\newaliascnt{assumption}{lemma}

\aliascntresetthe{assumption}

\newaliascnt{claim}{lemma}

\aliascntresetthe{claim}

\newaliascnt{notation}{lemma}

\aliascntresetthe{notation}

\newaliascnt{experiment}{lemma}

\aliascntresetthe{experiment}

\theoremsymbol{\ensuremath{\square}}

\newaliascnt{example}{lemma}
\newtheorem{example}[example]{Example}
\aliascntresetthe{example}

\newaliascnt{remark}{lemma}

\aliascntresetthe{remark}

\theoremseparator{:}
\theoremstyle{nonumberplain}
\theoremheaderfont{\normalfont\itshape}
\newtheorem{proof}{Proof}

\newcommand{\N}{\ensuremath{\mathbb{N}}}
\newcommand{\NN}{\ensuremath{\mathbb{N}_0}}

\newcommand{\R}{\ensuremath{\mathbb{R}}}
\newcommand{\C}{\ensuremath{\mathbb{C}}}

\newcommand{\abs}[1]{\ensuremath{\left\vert#1\right\vert}}

\newcommand{\dx}{\mathrm{d}}
\newcommand{\e}{\mathrm{e}}
\renewcommand{\i}{\mathrm{i}}
\newcommand{\zb}[1]{\ensuremath{\boldsymbol{#1}}}

\DeclareMathOperator*{\diag}{diag}

\DeclareMathOperator*{\argmin}{arg\,min}

\renewcommand{\Re}{\mathrm{Re}}
\renewcommand{\Im}{\mathrm{Im}}
\renewcommand{\d}{\, \mathrm{d}}

\newcommand{\norm}[1]{\left\lVert #1
  \right\rVert}

\newcommand\conj[1]{\overline{#1}}

\newcommand{\ba}{{\boldsymbol a}}

\newcommand{\cD}{{\mathcal D}}

\newcommand{\cR}{{\mathcal R}}

\newcommand{\tT}{\mathrm{T}}

\title{Time-Harmonic Optical Flow with \\
 Applications in Elastography}

\date{}
\author{%
  Oleh Melnyk\footnotemark[1]\\%
  {\footnotesize\href{mailto:melnyk@math.tu-berlin.de}{melnyk@math.tu-berlin.de}}
  \and
  Michael Quellmalz\thanks{Technische Universität Berlin, Institute of Mathematics, Straße des 17. Juni 136, D-10623 Berlin, Germany}\\%
  {\footnotesize\href{mailto:quellmalz@math.tu-berlin.de}{quellmalz@math.tu-berlin.de}}
  \and
  Gabriele Steidl\footnotemark[1]\\%
  {\footnotesize\href{mailto:steidl@math.tu-berlin.de}{steidl@math.tu-berlin.de}}
  \and
  Noah Jaitner\thanks{Charité Universitätsmedizin Berlin, Department of Radiology, Charitéplatz 1, D-10117 Berlin, Germany}\\%
  {\footnotesize\href{mailto:noah.jaitner@charite.de}{noah.jaitner@charite.de}}
  \and
  Jakob Jordan\footnotemark[2]\\%
  {\footnotesize\href{mailto:jakob.jordan@charite.de}{jakob.jordan@charite.de}}
  \and
  Ingolf Sack\footnotemark[2]\\%
  {\footnotesize\href{mailto:ingolf.sack@charite.de}{ingolf.sack@charite.de}}
}

\begin{document}

\def\sectionautorefname{Section}
\def\subsectionautorefname{Section} 

\maketitle

\begin{abstract}
In this paper, we propose
mathematical models for reconstructing the optical flow in time-harmonic elastography.
In this image acquisition technique, 
the object undergoes a special time-harmonic oscillation with known frequency so that only the spatially varying amplitude of the velocity field has to be determined.
This allows for a simpler multi-frame optical flow analysis using Fourier analytic tools in time. 
We propose three variational optical flow models 
and show how their minimization can be tackled via
Fourier transform in time.
Numerical examples with synthetic as well as 
real-world data demonstrate the benefits
of our approach.

\textit{Keywords:} optical flow, elastography, Fourier transform, iteratively reweighted least squares, Horn--Schunck method
\end{abstract}

\section{Introduction}

The tracking of motion along the sequence of image frames is a prominent problem in computer vision.
The main goal consists in finding the motion vector field between image frames of moving or deformed objects.
Optical flow methods trace back to the works of Horn--Schunck \cite{HornSchunck1981} and Lucas--Kanade \cite{LucasKanade1981}.
The former 
is an instance of variational methods, which have proven to be a powerful tool for this task.
The idea is to minimize a Tikhonov-type energy functional consisting of data fidelity and smoothness terms.
For an overview, we refer to the review papers \cite{FortunBouthemyKervrann2015,WeickertBruhnBroxPapenber2006,BePeSch15} as well as the performance comparisons \cite{BarronFleetBeauchemin1994}, and the more recent meta performance analysis \cite{,ZhaiXiangLvKong2021}.
Practical implementations of the optical flow estimation further rely on a combination of heuristics such as a coarse-to-fine pyramid, image warping, and inertial smoothing between iterations as highlighted in
the extensive experimental evaluations in \cite{SunRotBla14,SunRotBla10}.
Finally, neural network-based methods such as RAFT \cite{TD2020} and its multi-frame variants 
\cite{WCCLHHS2023} have shown state-of-the-art solutions for generalized optimal flow problems.
Motion magnification based on RAFT was studied in our previous work \cite{FHSS2023}.
An excellent overview of neural network-based solutions is presented in \cite{HuangGithub}.

In this paper, we consider a special optical flow problem for multiple image frames arising from optical time-harmonic elastography \cite{Jor23}. 
More precisely, an external acoustic source, 
whose frequency is precisely controlled by a wave generator, 
causes a time-harmonic oscillation of the specimen. 
Here, we assume that the flow velocity has the special structure
$v(t,x) = a(x) \Re(\e^{\i \omega t})\in\R^d$, which consists of a harmonic wave in time $t\ge0$ with known frequency $\omega$ and an amplitude $a(x)$ depending only on the spatial variable $x\in\R^d$.
Having reconstructed the velocity field $v$, it can be used in a subsequent step e.g. for the analysis of the material properties considered in the time-harmonic elastography experiment.
In particular, it is possible to reconstruct the shear modulus of the considered material
which is an important property that characterizes the mechanical behavior of the material, see \cite{Sack23}. 

Given a sequence of images of the specimen undergoing a harmonic deformation over multiple time steps $t$, we have only to reconstruct the time-independent amplitude~$a$
to get access to the displacement.
Therefore, we can efficiently utilize the image frames over the whole time period, i.e., a multiframe approach 
for reconstructing  the velocity field.
In contrast, most variational methods for general optical flow reconstruction rely just on two (consecutive) image frames. Nevertheless, there exist a few variational multiframe approaches, one of the first ones was that of Schn\"orr and Weickert \cite{WeickertSchnorr2001}.
Furthermore, a periodic model, which is not necessarily a single harmonic, was considered in the context of cardiac motion \cite{LiYang2010} and solved with the Horn--Schunck method.
This model was extended in \cite{Qi2011} to a 
smoothed version of the total variation norm and solved via gradient descent.
Due to their different nature, our model offers a lower numerical complexity, taking advantage of the fast Fourier transform and the fact that we need to reconstruct only a single amplitude image.

In this paper, we propose three models to solve the time-harmonic optical flow.
The first one is an adaption of the classical Horn--Schunck approach with quadratic data fidelity and smoothness term.  We give an explicit characterization of the minimizers of the respective energy functional via the Fourier transform in time. Based on this, we suggest solving it efficiently as a least squares problem.
The second model comprises an $L_1$ data fidelity term, making it more robust to noise, and a total variation smoothness term.
The third uses ``the best of both worlds'' by combining the data term of the second with the smoothness of the first model, therefore yielding a smoother amplitude field as it is expected from the physical point of view.
After reformulating in the Fourier domain, both models are solved with an iteratively reweighted least squares (IRLS) method \cite{Jung.2024,KumVerSto21}.
The accuracy and the lower arithmetic complexity of our algorithms are validated
in numerical experiments with both simulated and experimental data acquired in a gel phantom.



\paragraph{Outline of the paper.}
We start by providing our mathematical optical flow model in time-harmonic elastography in Section \ref{sec:OF}. Then, in Section \ref{sec:var}, we propose three variational models to approximate the optical flow. We derive formulas
for finding the minimizing amplitude of the optical flow field using Fourier transforms.
Implementation details are discussed in Section~\ref{sec:numerics_real}.
Both synthetic and real-world examples which demonstrate the very good numerical performance of our approach are given in 
Section~\ref{sec:numerics}.
Final conclusions are drawn in Section~\ref{sec:conclusion}.

\section{Time-harmonic optical flow} \label{sec:OF}
In this section, we provide the continuous optical flow model for the special setting 
when the object undergoes a time-harmonic oscillation with known frequency.

Let $T>0$ be the observational time and $\Omega$ be a bounded domain in $\R^d$.
We are given gray-value images $I\colon [0,T] \times \Omega  \to \R$ in $C^1$, which means that they are continuously differentiable in both time and space.
In our applications, we deal with sequences of two-dimensional images, i.e.\ $d=2$ and $\Omega$ is a rectangular domain.
From the experimental setting, we can assume that the so-called ''brightness constancy'' assumption holds true, meaning that the image gray-values (pixel intensities)
move in space, but do not change their values.
Roughly speaking,
we assume for all $x \in \Omega$ that 
\begin{equation} \label{eq:bca}
I(t,\varphi(t,x)) = I(0,x) \quad \text{for all} \quad t \in [0,T],
\end{equation}
where $\varphi\colon [0,T] \times \Omega  \to \R^d$ fulfills for every $x \in \Omega$ the ordinary differential equation (ODE)
\begin{equation}  \label{eq:dphi}
\begin{aligned}
\frac{\dx}{\dx t} \varphi (t,x) &= v(t, \varphi(t,x)), \quad t \in [0,T],\\
\varphi (0,x) &= x,
\end{aligned}
\end{equation}
with a sufficiently smooth velocity field $v\colon [0,T] \times \Omega  \to \R^d$.
Then \eqref{eq:bca} can be rewritten as
\begin{equation} \label{eq:of}
0 = \frac{\dx}{\dx t} I(t,\varphi(t,x)) 
= \partial_t I(t, \varphi (t,x)) + \nabla I(t,\varphi (t,x))^\tT \partial_t \varphi (t,x),
\end{equation}
where $\nabla$ denotes the gradient with respect to the spatial component in $\R^d$ and ${\,}^\tT$ is the transposition of a matrix.
By \eqref{eq:dphi}, and replacing $ \varphi (t,x) \to x$ for every $t \in [0,T]$,
the image sequence fulfills the partial differential equation
\begin{equation}  \label{eq:flow}
\begin{aligned}
\partial_t I(t,x) + \nabla I(t,x)^\tT v(t,x)
&= 0, \quad &(t,x) &\in [0,T] \times \Omega,\\
I (0,x) &= I_0(x), \quad &x &\in \Omega, 
\end{aligned}
\end{equation}
where $I_0$ is the first image frame.

For the time-harmonic optical flow elastography \cite{Jor23}, the object undergoes a periodic, time-harmonic displacement with frequency $\omega\in\R$, which is known from the experimental setup. 
This implies that the displacement $u\colon [0,T] \times \Omega \to \R^d$ in
$$
\varphi(t,x) = x + u(t,x)
$$
is time-harmonic, i.e.,
\begin{equation} \label{eq:u-harmonic}
u(t,x) = \Re 
( \tilde a(x) \, \e^{\i t \omega} ) - \Re (\tilde a(x)),
\end{equation}
with a complex amplitude 
$
\tilde a \in\C$, where the term $\Re( \tilde a(x))$ comes from the second line of \eqref{eq:dphi}. 
However, it appears to be hard to work with this approach in \eqref{eq:flow}, because the reconstruction of $\tilde a$ requires solving a coupled system of the two differential equations \eqref{eq:dphi} and \eqref{eq:flow}.
Therefore, we propose to use instead the ''ansatz''
\begin{equation}\label{eq:va}
\begin{aligned} 
v(t,x) 
&= \Re \left( a(x) \, \e^{\i t \omega} \right)\\
&=
a_{\text{R}}(x) \cos(t \omega) - a_{\text{I}}(x) \sin(t \omega)
,\qquad t\in[0,T],\  x\in\Omega,
\end{aligned}
\end{equation}
where the amplitude $a = a_{\text{R}} + \i a_{\text{I}}
$ with $a_{\text{R}}, a_{\text{I}} \in\R$ is again complex and we always assume that $a$ is $C^2$.
The advantage of this approach is that the systems \eqref{eq:dphi} and \eqref{eq:flow} are no longer coupled. We first reconstruct $a$ as a solution of \eqref{eq:flow}, where we substitute $v$ by \eqref{eq:va}, and subsequently obtain the displacement~$u$ by solving \eqref{eq:dphi}. The relation between both approaches is discussed in the following paragraph.

\paragraph{Time-harmonic approach with $v$ versus $u$.}
Let us consider the relation in the one-dimensional case $d=1$.
In particular, we are interested in the question if 
the specific time-harmonic structure \eqref{eq:va} of the velocity $v$ 
implies that the deformation $\varphi$ is also periodic 
or even that $u$ is time-harmonic. 
We will see that the first property is fulfilled at least for some choices of the amplitude $a$.
For the special amplitude function 
\begin{equation} \label{aha}
  a(x) =  \alpha (x) \e^{\i\phi_0} 
\end{equation}
with a  locally Lipschitz continuous real-valued function $\alpha$ 
and the corresponding velocity  
$$v(t,x) = \alpha(x) \cos(t\omega + \phi_0),$$ 
$\varphi$ has to fulfill
\begin{equation}  \label{eq:dphi-alpha}
\frac{\dx}{\dx t} \varphi(t,x) = \alpha \left(\varphi(t,x) \right) \cos(t\omega + \phi_0)
,\quad \varphi (0,x) = x.
\end{equation}
This initial value problem can be solved by the separation of variables resulting for every $x \in \R$ in a solution $\varphi(t,x)$ which is periodic in $t$.
In particular, if $1/\alpha$ has a bijective antiderivative $A\colon\R\to\R$, which is e.g.\ the case when $\alpha$ is bounded and bounded away from zero, then the solution is 
$$
\varphi(t,x)
=
A^{-1}\left( A(x) - \frac{\sin(t\omega+\phi_0) - \sin(\phi_0)}{\omega} \right).
$$
A first order Taylor approximation of $A^{-1}$ around $A(x)$ yields the approximate displacement
$$
u(t,x)
=
\varphi(t,x) - x
\approx
-\alpha(x) \frac{\sin(t\omega+\phi_0) - \sin(\phi_0)}{\omega},
$$
which is a harmonic oscillation of the form \eqref{eq:u-harmonic}.

\begin{figure}[!t] \centering
 \newcommand{\cc}{0.2}
\tikzset{font=\small}
	\begin{tikzpicture}
	\begin{axis}[xlabel=$t$, 
	width=.45\textwidth, height=0.32\textwidth, xmin=0, xmax=6.28,
	xtick={0,1.57,3.14,4.71, 6.28},
	xticklabels={0,$\pi/2$,\vphantom{1}$\pi$,$3\pi/2$,$2\pi$},
    every axis y label/.style={at={(ticklabel* cs:1.03)}, anchor=south,},
    cycle list name=exotic, 
	legend pos=outer north east, legend style={cells={anchor=west}}]
    \pgfplotsinvokeforeach{3, ..., -3} {
	\addplot+ [mark=none, thick,smooth, samples=50, domain=0.001:6.5] {#1*0.4*exp(\cc*sin((180*x)/pi))};
    \addlegendentry{$\varphi(t,\pgfmathparse{#1*0.4}\pgfmathprintnumber\pgfmathresult)$}; }
	\end{axis}
	\end{tikzpicture} 
    \caption{Trajectories $\varphi(t,x)$ corresponding to the velocity amplitude $a(x)=\cc\,x$ from \autoref{ex:periodic} with different starting values $x$ for $\omega=1$.\label{fig:traj-v}}
\end{figure}
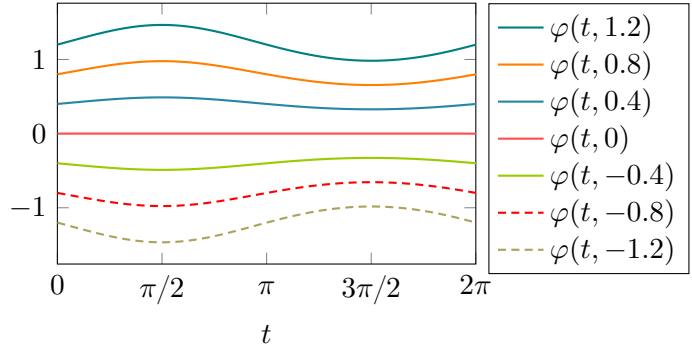

\begin{figure}[!b]
    \centering
    \includegraphics[width = 1\textwidth]{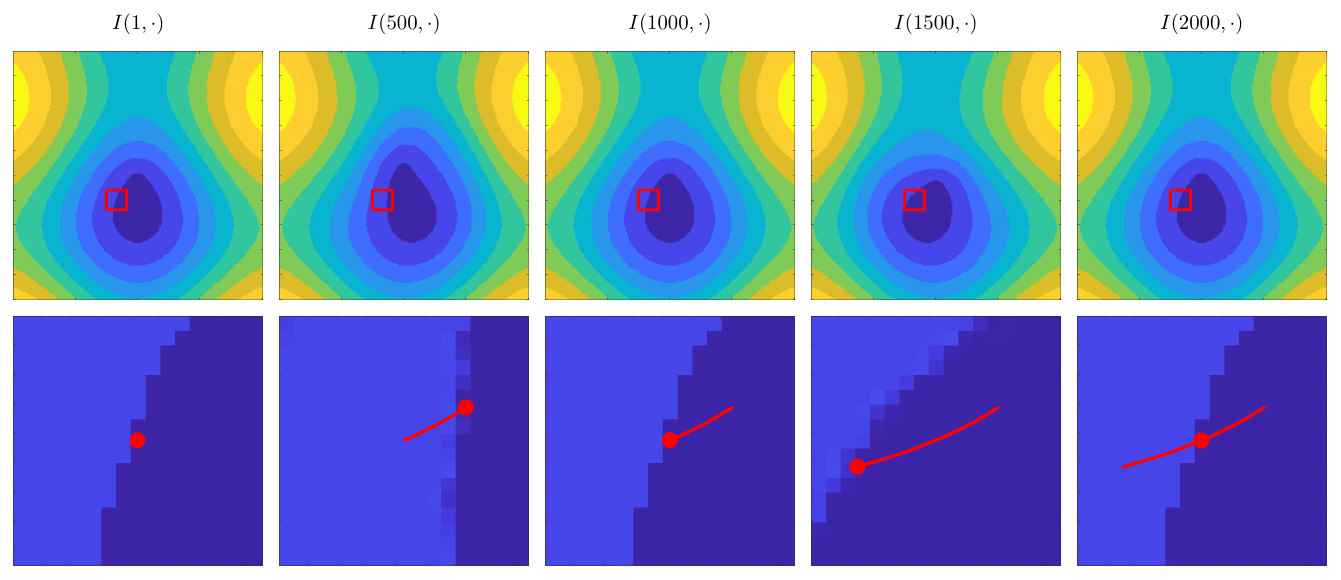}
    \caption{Two-dimensional images $I(t,\cdot)$ corresponding to a harmonic velocity \eqref{eq:va} with a real-valued amplitude function $a$ and period $T=2000$. The displacement $u$ is periodic. The second row is the zoom-in of the first row to the red rectangle.
    The red $\color{red}\pmb\bullet$ depicts the position $\varphi(t,x)$ of a single point $x$ at time $t$ and the red curve is its trajectory.}
    \label{fig:2d disp periodic}
\end{figure}

\begin{example}[Periodic deformation] \label{ex:periodic}
Let $a(x) = cx$, $x\in\R$, for some $c\in\R$. Then the solution of \eqref{eq:dphi} becomes
\begin{equation}
\varphi(t,x)
=
x\,\e^{-\frac{c}{\omega}\sin(t\omega)},
\end{equation}
which is illustrated in \autoref{fig:traj-v}.
In particular, the deformation $\varphi$ is again $\frac{2\pi}{\omega}$ periodic in~$t$.
%
An illustration of a two-dimensional example, where the resulting displacement $u$ is periodic, is given in \autoref{fig:2d disp periodic}.
\end{example}

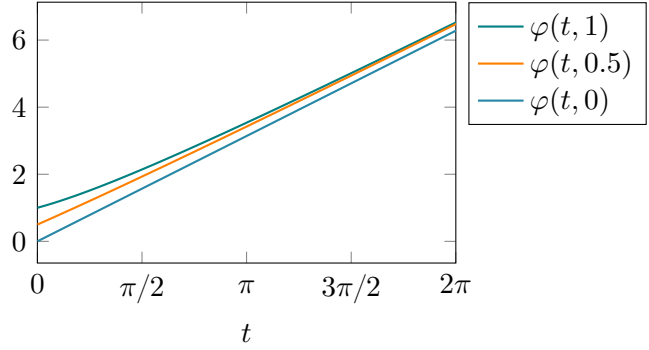
\begin{figure}[!t]\centering
\tikzset{font=\small}
	\begin{tikzpicture}
	\begin{axis}[xlabel=$t$, 
	width=.45\textwidth, height=0.32\textwidth, xmin=0, xmax=6.28,
	xtick={0,1.57,3.14,4.71, 6.28},
	xticklabels={0,$\pi/2$,$\vphantom{1}\pi$,$3\pi/2$,$2\pi$},
    every axis y label/.style={at={(ticklabel* cs:1.03)}, anchor=south,},
    cycle list name=exotic, 
	legend pos=outer north east, legend style={cells={anchor=west}}]
 
    \pgfplotsinvokeforeach{2,...,0} {
	\addplot+ [mark=none, thick,smooth, samples=50, domain=0.001:6.5] {x+pi-2*rad(atan(x+cot(deg(#1/2/2))))};
    \addlegendentry{$\varphi(t,\pgfmathparse{#1/2}\pgfmathprintnumber\pgfmathresult)$}; }

	\end{axis}
	\end{tikzpicture}\hfill
 
    \caption{Trajectories $\varphi(t,x)$ of \autoref{ex:nonperiodic} for different starting values $x$ with $\omega=1$.\label{fig:nonperiodic}}
\end{figure}

\begin{figure}[!b]
    \centering
    \includegraphics[width = 1\textwidth]{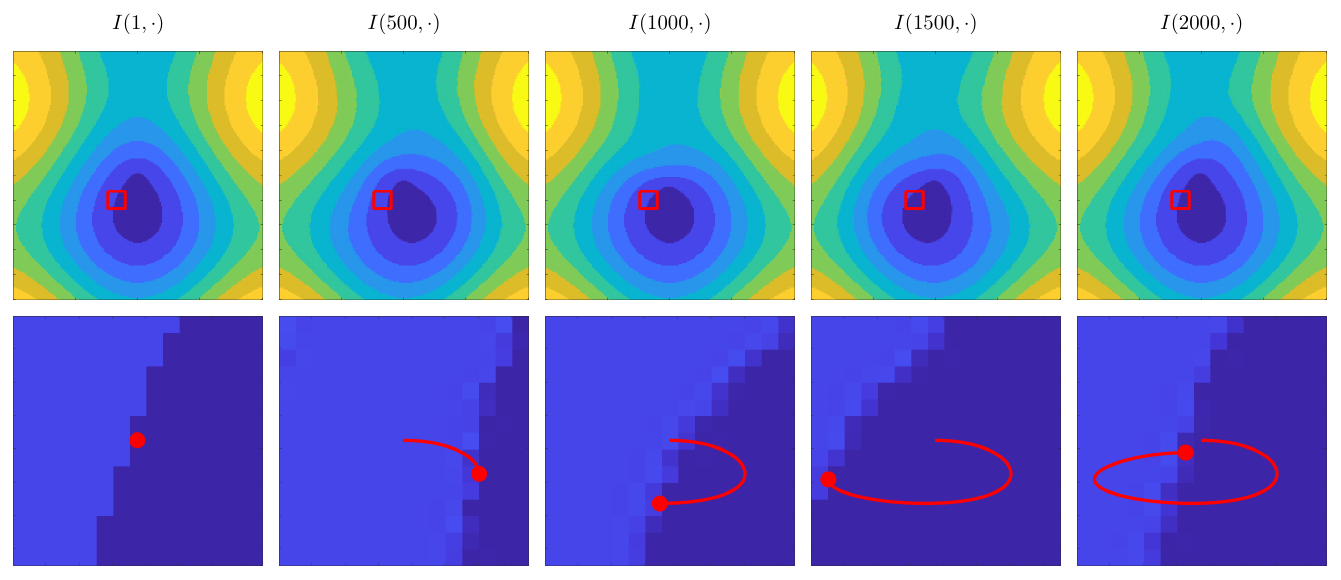}
    \caption{Two-dimensional images $I(t,\cdot)$ corresponding to a harmonic velocity \eqref{eq:va} with an amplitude function of the form $a(x)\in\i\R\times\R$ and period $T=2000$. The resulting displacement $u$ is not periodic. The second row is the zoom-in of the first row to the red rectangle.
    The red $\color{red}\pmb\bullet$ depicts the position $\varphi(t,x)$ of a single point $x$ at time $t$ and the red curve is its trajectory.}
    \label{fig:2d disp nonperiodic}
\end{figure}

\begin{example}[Non-periodic deformation] \label{ex:nonperiodic}
  If the amplitude $a$ does not have the form \eqref{aha}, the displacement $u$ of the motion might be non-periodic.
  Consider the complex amplitude $a(x)=\e^{-\i x}$, $x\in\R$, and frequency~$\omega=1$.
  Then the velocity is
  $$
  v(t,x)
  =
  \Re(\e^{\i(t- x)})
  =
  \cos(t-x).
  $$
  A straightforward computation using trigonometric identities verifies that
  $$
  \varphi(t,x)
  =
  t + 2 \operatorname{arccot}\left(t+\cot\tfrac{x}{2}\right)
  $$
  for $x\in(0,\pi)$ is a solution of \eqref{eq:dphi} in this case, in particular we have
  $$
  \frac{\dx}{\dx t} \varphi(t,x)
  =
  1-\frac{2}{1+(t+\cot\frac{x}{2})^2}
  =
  \cos(t-\varphi(t,x)).
  $$
  Due to the asymptotics of the inverse cotangent, we have $\varphi(t,x)\to\infty$ for $t\to\infty$ and fixed $x$, cf.\ \autoref{fig:nonperiodic}. 
  Hence, the same holds also for the displacement $u(t,x)$.

  An illustration of a two-dimensional example, where the resulting displacement $u$ is non-periodic, is given in \autoref{fig:2d disp nonperiodic}.
\end{example}

An ``opposite'' example showing a time-harmonic displacement $u$ that results in a periodic
velocity $v$, which is however not time-harmonic, is given in Appendix \ref{app:another_ex}.

\section{Variational models for time-harmonic optical flow} \label{sec:var}
We want to approximate the solution of the flow equation \eqref{eq:flow} with the time-harmonic velocity field \eqref{eq:va}
from given measurements $I(t,x)$, $t \in [0,T]$, $x \in \Omega$, where 
we assume that $p \in \N$ periods of oscillations were observed and $T \coloneqq \frac{2\pi p}{\omega}$.
More precisely, we aim to recover the amplitude $a$ of the velocity field $v$. 
Since the optical flow equation is ill-posed, this is in general tackled 
by solving a variational problem of the form
\begin{equation} \label{eq:HS-general}
E(v)
\coloneqq
\int_{\Omega} \int_{0}^{T}
\Big( \cD
\big(  \underbrace{\nabla I(t,x)^\tT v(t,x)
+ \partial_tI(t,x) }_{=:G_v(t,x)} \big)
+ \lambda \cR(\nabla v(t,x))
\Big) \d t
\d x, 
\end{equation}
with a \emph{data fidelity term} $\cD\colon \R\to\R_{\ge0}$ enforcing the physics-based relation \eqref{eq:flow} 
and a \emph{regularizer} $\cR\colon \R^{d\times d}\to\R_{\ge0}$, which aims to make the
problem well-posed and contains, from a Bayesian point of view, prior knowledge on the velocity field.
Both terms are coupled by a regularization parameter $\lambda>0$.
For an overview of data fidelity and regularization terms used in optical flow between two image frames, see e.g.\ \cite{FortunBouthemyKervrann2015}.
In this section, we will discuss the following choices: 
\begin{itemize}
\item[I)] $\cD (G_v) \coloneqq |G_v|^2$ and $\cR \coloneqq \|\nabla v
\|^2 = \sum_{j,k=1}^d \abs{\partial_{x_j} v_k}^2$,
\item[II)] $\cD (G_v) \coloneqq |G_v|$ and $\cR \coloneqq \|\nabla v\|$,
\item[III)] $\cD (G_v) \coloneqq |G_v|$ and $\cR \coloneqq \|\nabla v\|^{2}$,
\end{itemize}
where $\|\cdot\|$ denotes the Euclidean norm in $\R^{d \times d}$, also known as Frobenius norm.
Model I resembles the Horn--Schunck model in image processing \cite{HornSchunck1981}
and is appropriate for image data corrupted by Gaussian noise and a sufficiently smooth
velocity field. 
Model II is more robust against noise in the image frames, especially  heavy-tailed noise stemming, e.g., 
from the Laplace distribution
and also avoids penalizing jumps
in the derivative of the velocity field too much, see the total variation model in image denoising \cite{ROF1992} and for its counterpart in optical flow, e.g.\ \cite{ADK1999, HSSW2002}.
Such models appeared to be beneficial for early crack detection in materials
in tensile tests \cite{BalBecEifFitSchSte19}.
Note that often the above model was only considered for two consecutive image frames
and not over the whole sequence of images. 
Finally, Model III combines the robustness to noise from Model II with the regularization term of Model I and aims for a smoother reconstruction.

In the following, we will show how the Models I -- III can be 
refined using the special time-harmonic structure of the velocity field in \eqref{eq:va}.
In particular, the incorporation of the whole time interval, i.e. of all image frames at the same time, becomes straightforward. To this end, we need 
the {Fourier transform} of  functions $f \in L_1(\R)$ 
defined by
\begin{equation} \label{eq:FT}
\mathcal F[f](\nu)
\coloneqq
\int_{\R}  f(t)\, \e^{-\i t \nu} \d t
,\qquad \nu\in\R.
\end{equation}
If $f$ is supported on $[0,T]$ as it will be the case in our applications, 
the integral can be just taken over $[0,T]$.
For a vector-valued function $f$, 
the Fourier transform is defined componentwise. In what follows, we always apply the Fourier transform in the time variable $t$ and leave the spatial variables fixed.

\subsection{Model I} \label{sec:L2L2}
We start by minimizing the functional 
\begin{equation} \label{eq:HS}
E(v)
\coloneqq
\int_{\Omega} \int_{0}^{T}
\left( 
\abs{ G_v(t,x) }^2
+ \lambda \norm{\nabla v(t,x)}^2
\right) \d t
\d x ,
\end{equation}
where we are only interested the special velocity fields $v$ in \eqref{eq:va}.
Then the minimizer of $E(v)$ is described by the following theorem.

\begin{theorem}\label{thm: Horn-Schunck harmonic}
For a given frequency $\omega \in \R$ and $p \in \N$, set 
$T \coloneqq 2\pi p /\omega$.
Let $I\colon [0,T] \times \Omega \to \R$ be an image sequence in $C^1$ and
$v\colon [0,T] \times \Omega  \to \R^d$ a time-harmonic velocity in $C^2$ of the form \eqref{eq:va}. 
Then, the minimizing amplitude $a$ corresponding to the velocity $v$ in \eqref{eq:HS} 
satisfies for all $x \in \Omega$ the equation
\begin{equation} \label{eq: linear system corrected}
\lambda T
\Delta a(x) 
= \mathcal F[\nabla I \nabla I^\tT] (2\omega,x) \overline{a(x)}
+ \mathcal F[\nabla I \nabla I^\tT] (0,x)
a(x)
+ 2
\mathcal F[\partial_t I \, \nabla I](\omega, x) ,
\end{equation}
where $\overline{a}$ denotes the complex conjugate of $a$ and $\Delta$ is the Laplace operator.
\end{theorem}

\begin{proof}
Inserting \eqref{eq:va}
into $E$ and splitting the amplitude $a = a_{\text{R}} +  a_{\text{I}}$ as in \eqref{eq:va}, we obtain 
\begin{align*}
E(v)
& =
\int_{\Omega}
\int_{0}^{T}
\left( 
\left( \nabla I(t,x)^\tT a_{\text{R}}(x) \cos(t \omega) - \nabla I(t,x)^\tT a_{\text{I}}(x) \sin( t \omega)  
+ \partial_tI(t,x) \right)^2 \right. \\
& \left. \qquad \qquad + \lambda \norm{\nabla a_{\text{R}}(x) \cos(t \omega) - \nabla a_{\text{I}}(x) \sin( t \omega)}^2
\right)
\d t
\d x .
\end{align*}
Since we integrate over $p$ full periods, this becomes
\begin{align}
E(v)
& =
\int_{\Omega}
\left( 
\int_{0}^{T}
\left( \nabla I(t,x)^\tT a_{\text{R}}(x) \cos(t \omega) - \nabla I(t,x)^\tT a_{\text{I}}(x) \sin(t \omega)  
+ \partial_t I(t,x) \right)^2 \right. 
\d t \nonumber \\
& \left. \qquad \qquad + \frac{\lambda T}{2} \left(\norm{\nabla a_{\text{R}}(x)}^2 + \norm{\nabla a_{\text{I}}(x)}^2\right)
\right)
\d x. \label{eq: regularizer simplified}
\end{align}
The Euler--Lagrange equations 
yield that the minimizer $a$ has to satisfy for all $i=1,\dots,d$ the equations
\begin{align*}
\lambda T \sum_{j=1}^{d}
\partial^2_{x_j} (a_{\text{R}})_i(x) 
= 
\int_{0}^{T} 2 \Big( & \nabla I(t,x)^\tT  a_{\text{R}}(x) \cos(t \omega) - \nabla I(t,x)^\tT a_{\text{I}}(x) \sin(t \omega)  
\\&
+ \partial_tI(t,x)  \Big)\,
\partial_{x_i} I(t,x) \cos(t \omega) \d t,
\\
\lambda T \sum_{j=1}^{d}
\partial^2_{x_j} (a_{\text{I}})_i(x)  
= 
\int_{0}^{T} 2 \Big( & \nabla I(t,x)^\tT  a_{\text{R}}(x) \cos(t \omega) - \nabla I(t,x)^\tT a_{\text{I}}(x) \sin(t \omega)  
\\&
- \partial_tI(t,x)  \Big)\,
\Big( - \partial_{x_i} I(t,x) \sin(t \omega) \Big) \d t .
\end{align*}
Using trigonometric identities $2 \cos^2(y) = 1 + \cos(2y)$ and $2 \sin^2(y) = 1 - \cos(2y)$, we can equivalently state the last equations in vector form as
\begin{align*}
\lambda T \Delta a_{\text{R}}(x)
= 
\int_{0}^{T}\nabla I(t,x)\, \Big( & \nabla I(t,x)^\tT a_{\text{R}}(x) 2 \cos^2(t \omega)
\\&
- \nabla I(t,x)^\tT a_{\text{I}}(x) \sin(2 t \omega)
+ \partial_tI(t,x) 
2 \cos(t \omega) \Big) \d t 
\\
\lambda T \Delta a_{\text{I}}(x)
=
\int_{0}^{T}\nabla I(t,x)\, \Big(& - \nabla I(t,x)^\tT a_{\text{R}}(x) \sin(2 t \omega)
\\&
+ \nabla I(t,x)^\tT a_{\text{I}}(x) 2 \sin^2(t \omega)
- \partial_tI(t,x) 
2 \sin(t \omega) \Big) \d t.
\end{align*}
Then, the temporal Fourier transform \eqref{eq:FT} yields 
\begin{equation*}
\begin{aligned}
\lambda T \Delta a_{\text{R}}(x)
={}&
\left(\Re \mathcal F[\nabla I \nabla I^\tT](0,x) + \Re \mathcal F[\nabla I \nabla I^\tT](2\omega, x) \right) a_{\text{R}}(x)
\\&
+  \Im \mathcal F[\nabla I \nabla I^\tT](2\omega, x) a_{\text{I}}(x)
+ 2 \Re \mathcal F[\partial_t I \nabla I](\omega, x),
\\
\lambda T \Delta a_{\text{I}}(x)
={}&
\left(\Re \mathcal F[\nabla I \nabla I^\tT](x, 0) - \Re \mathcal F[\nabla I \nabla I^\tT](2\omega, x) \right) a_{\text{I}}(x)
\\&
+ \Im \mathcal F[\nabla I \nabla I^\tT](2\omega, x) a_{\text{R}}(x)
+ 2 \Im \mathcal F[\partial_t I \nabla I](\omega, x).
\end{aligned}
\end{equation*}
Since $I(t,x) \in \R$, we get $\mathcal F[\nabla I \nabla I^\tT](x, 0) \in \R^{d \times d}$. By returning 
to the complex-valued function $a$, 
we get the assertion.
\end{proof}

As the unknown $a$ depends only on the spatial variable $x$ and the Fourier coefficients of the derivatives of $I$ can be precomputed, the size of the linear system \eqref{eq: linear system corrected} is independent of the time period $T$. This causes a significant runtime improvement compared to the estimation of the optical flow velocity between all consecutive pairs of images.

\subsection{Model II} \label{sec:L1L1}

For Model II, 
the functional \eqref{eq:HS-general} becomes
\begin{equation} \label{eq:HS-l1l2}
E(v)
\coloneqq
\int_{\Omega} \int_{0}^{T}
\left(  \abs{G_v(t,x) }
+ \lambda \norm{\nabla v(t,x)}
\right) \d t \d x. 
\end{equation}
The data fidelity term resembles the $L_1$ norm on $[0,T]\times\Omega$ and is also known as \emph{least absolute deviations} \cite{BloSte84}.
A modified $L_1$ term has been applied in optical flow estimation between two images \cite{BroBruPapWei04}. In particular, they used the Charbonnier
penalty $\cD(G_v) = \sqrt{|G_v|^2+\varepsilon^2}$ for some $\varepsilon>0$ as a part of their data fidelity term.
The regularizer 
can be seen as the vectorial \emph{total variation} (TV) norm of $v$, see \cite[eq.\ (9)]{GolCre10}.

As \eqref{eq:HS-l1l2} is no longer differentiable, finding its minimizer via Euler--Lagrange equations is not possible. While a variety of optimization techniques for \eqref{eq:HS-l1l2} exist, we chose to work with a majorize--minimize technique known as the \emph{iteratively reweighted least squares} (IRLS) \cite{KumVerSto21, Jung.2024} or half-quadratic minimization \cite{SteidlPerschHielscherChanBergmann2016}. For each iteration of the algorithm, a global quadratic majorant of \eqref{eq:HS-l1l2} depending on the current iterate is constructed. Then it is minimized and its optimizer is set to be the next iterate. 

To construct a majorant of \eqref{eq:HS-l1l2}, we consider the \emph{Huber function}  
\[
h_{\varepsilon}\colon\R\to [\tfrac\varepsilon2,\infty),
\quad x \mapsto \begin{cases}
|x|, & |x| \ge \varepsilon, \\
\frac{|x|^2}{2 \varepsilon} + \frac{\varepsilon}{2}, & |x| < \varepsilon,
\end{cases}
\]
with smoothing parameter $\varepsilon>0$. It is a smoothed version of $|x|$ linked to its Moreau envelope \cite[Examples 6.54 and 6.62]{beck2017first} and satisfies the following properties. 

\begin{lemma}\label{l: Huber function}
Let $x,z \in \R$. For all $0 < \varepsilon_1 \le \varepsilon_2$, we have
\[
h_{\varepsilon_1}(x) \le h_{\varepsilon_2}(x)
\quad \text{and} \quad
\lim_{\varepsilon \downarrow 0} h_\varepsilon(x) = |x|.
\]
\end{lemma}

Next, we replace $\abs{\cdot}$ by $h_\varepsilon$ in the energy functional \eqref{eq:HS-l1l2}. 
We consider the smoothed energy functional 
\[
E_{\varepsilon,\delta}(v)
\coloneqq \int_{\Omega} \int_{0}^{T}
\left( h_\varepsilon\left(G_v (t,x)) \right)
+ \lambda h_{\delta}\left(\norm{\nabla v(t,x)} 
\right) \right) \d t
\d x 
\ge E(v),
\]
with data fidelity and regularization smoothing parameters $\varepsilon,\delta>0$. It is a majorant of $E(v)$, however, it is not a quadratic functional. Thus, we proceed further by utilizing the following lemma. 

\begin{lemma}[{{\cite[Lem. 2.1]{KumVerSto21}}}]\label{l: Huber function quadratic majorant}
Let $x,z \in \R$ and $\varepsilon>0$. Define 
\begin{equation}\label{eq: majorizer q def}
q_\varepsilon(x, z)
\coloneqq \frac{|x|^2}{2 \max\{ \varepsilon, |z| \} } + \frac{1}{2}\max\{\varepsilon, |z|\}.
\end{equation}
Then, $q_\varepsilon$ is a quadratic majorant of $h_\varepsilon$ satisfying 
\[
h_\varepsilon(x) \le q_\varepsilon(x,z)
\quad \text{and} \quad  
h_\varepsilon(x) = q_\varepsilon(x,x).
\]
\end{lemma}

For $u\colon[0,T]\times\Omega\to\R^d $, we set
\begin{equation}\label{eq:Eved}
E_{u, \varepsilon, \delta}(v)
\coloneqq
\int_{\Omega} \int_{0}^{T}
\left( q_\varepsilon\left( G_v (t,x), G_{u} (t,x) \right)
+ \lambda q_{\delta}\left(\norm{\nabla v(t,x)}  , \norm{\nabla u (t,x)} 
\right) \right) \d t
\d x.
\end{equation}
Then we have
\[
E(v) \le E_{\varepsilon, \delta}(v) \le E_{u, \varepsilon, \delta}(v).
\]
The iteration
\[
v^{k+1} = \argmin_{v} E_{v^k, \varepsilon, \delta}(v)
,\qquad k\in\NN,
\]
leads to
\begin{equation}\label{eq: IRLS weak convergence}
E_{\varepsilon, \delta}(v^{k+1}) \le E_{v^k, \varepsilon, \delta}(v^{k+1}) \le E_{v^k, \varepsilon, \delta}(v^k) = E_{\varepsilon, \delta}(v^{k})
\end{equation}
and the values $ E_{\varepsilon, \delta}(v^{k}) $ form a nonincreasing convergent sequence for $k\to\infty$.
For fixed parameters $\varepsilon$ and $\delta$, this procedure only minimizes $E_{\varepsilon, \delta}$ instead of \eqref{eq:HS-l1l2}. Since $E_{\varepsilon, \delta}$ goes to $E$ as the smoothing parameters vanish, a common strategy is to take sequences $ \varepsilon_k,\delta_k\downarrow0$ and construct the iterates 
\[
v^{k+1} \coloneqq \argmin_{v} E_{v^k, \varepsilon^k, \delta^k}(v).
\]
This step does not affect the convergence argument in \eqref{eq: IRLS weak convergence} as 
\[
E_{\varepsilon^{k+1}, \delta^{k+1}}(v^{k+1})
\le E_{\varepsilon^{k}, \delta^{k}}(v^{k+1})
\le E_{\varepsilon^{k}, \delta^{k}}(v^{k}).
\]
A more advanced derivation \cite[Thm. 6]{Jung.2024} shows that (in a slightly different setting) with a proper decay strategy of $\varepsilon^{k+1}, \delta^{k+1}$, the iterates $v^k$ are guaranteed to converge to a global minimum of $E(v)$ at a sublinear rate $\mathcal O(k^{-\frac{1}{3}})$. Furthermore, with an additional assumption, a linear convergence is achievable \cite[Thm. 9]{Jung.2024}.

Now we focus on the minimization of $E_{v^k, \varepsilon^k, \delta^k}$. In fact, it can be seen as a weighted version of \eqref{eq: linear system corrected}.

\begin{theorem} \label{thm:L1L1}
Let $\varepsilon, \delta > 0$ and $u\colon[0,T]\times\Omega\to\R^d$ in $C^1$ be fixed. Consider the flow velocity $v(t,x) = \Re(a(x) \e^{\i \omega t})$ in $C^2$.  Then the amplitude $a\colon\Omega\to\C^d$ of the minimizer~$v$ of $E_{u, \varepsilon, \delta}(v)$, see \eqref{eq:Eved}, satisfies
\begin{equation} \label{eq: linear system weighted}
\begin{aligned}
& \mathcal F[w_\cD \nabla I \nabla I^\tT](0,x) a(x) + \mathcal F[w_\cD \nabla I \nabla I^\tT](2\omega, x) \overline{a (x)} 
+ 2 \mathcal F[w_\cD \partial_t I \nabla I](\omega, x) \\
&\quad = \lambda  \sum_{j=1}^d \partial_{x_j} \left( \mathcal F[w_\cR](0,x) \partial_{x_j} a(x) + \mathcal F[w_\cR](2\omega, x) \overline{\partial_{x_j} a (x)} \right) 
\end{aligned}
\end{equation}
for all $x\in\Omega$,
where $w_\cD, w_\cR \colon [0,T]\times\Omega\to \R$ are defined by
\begin{equation}\label {eq:wRD}
w_\cD(t,x) \coloneqq \frac{1}{\max\{\varepsilon, |G_u (t,x)| \}}
\quad \text{and} \quad
w_\cR(t,x) \coloneqq \frac{1}{\max\{\delta, \norm{\nabla u (t,x)}  \}}.
\end{equation}
\end{theorem}
\begin{proof}
As the second term in the definition \eqref{eq: majorizer q def} of $q_\varepsilon(x,z)$ does not depend on $x$, we can write $E_{u, \varepsilon, \delta}(v)$ as
\begin{align*}
E_{u, \varepsilon, \delta}(v) 
& = 
\frac{1}{2}\int_{\Omega} \int_{0}^{T}
\left( \frac{ |G_v (t,x)|^2 }{ \max\{\varepsilon, |G_u (t,x)| \} }
+ \lambda \frac{ \norm{\nabla v(t,x)} ^2}{ \max\{ \delta, \norm{\nabla u (t,x)}  \} }
\right) \d t 
\d x
+ \textrm{const} \\
& 
= 
\frac{1}{2}\int_{\Omega} \int_{0}^{T} w_\cD(t,x) |G_v (t,x)|^2 
+ \lambda w_\cR(t,x) \norm{\nabla v(t,x)} ^2
\d t 
\d x
+ \textrm{const.}
\end{align*}
We first want to compute the regularization term
$$
R(a(x))
\coloneqq
\int_0^T w_\cR(t,x) \norm{\nabla v(t,x)} ^2 \d t.
$$
Substituting $v(t,x) = \Re(a(x) \e^{\i \omega t})$ gives
\begin{align*}
\norm{\nabla v(t,x)} ^2 
& = \norm{ \Re(\nabla a(x) \e^{\i \omega t}) } ^2  \\
& = \tfrac{1}{4} \norm{ \nabla a(x) \e^{\i \omega t} + \conj{\nabla a}(x) \e^{-\i \omega t} } ^2 \\
& = \tfrac{1}{2} \norm{ \nabla a(x) } ^2 + \tfrac{1}{2} \Re \left( \left\langle \nabla a(x) \e^{\i \omega t}, \conj{\nabla a}(x) \e^{-\i \omega t}  \right \rangle_{\mathrm{Fro}} \right) \\
& = \frac{1}{2} \norm{ \nabla a(x) } ^2 + \frac{1}{2} \Re \left( \sum_{k,j=1}^d (\partial_{x_j} a_k (x))^2 \e^{2\i \omega t} \right).
\end{align*}
Integrating over $t$ yields
\begin{align*}
R(a(x))
&= \frac{1}{2} \int_0^T w_\cR(t,x) \left[ \norm{ \nabla a(x) } ^2 + \Re \left( \sum_{k,j=1}^d (\partial_{x_j} a_k (x))^2 \e^{2\i \omega t} \right) \right] \\
& =\frac{1}{2} \mathcal F[w_\cR](0,x) \norm{ \nabla a(x) } ^2 + \frac{1}{2} \Re \left( \sum_{k,j=1}^d (\partial_{x_j} a_k (x))^2 \mathcal F[w_\cR](x,-2\omega) \right) \\
& =\frac{1}{2} \mathcal F[w_\cR](0,x) \norm{ \nabla a(x) } ^2 + \frac{1}{2} \Re \left( \sum_{k,j=1}^d (\partial_{x_j} a_k(x))^2 \conj{\mathcal F[w_\cR]}(x,2\omega) \right),
\end{align*}
where in the last line we used that the temporal Fourier transform of the real-valued $w_\cR$ is conjugate symmetric. Splitting $a = a_{\text{R}} + \i a_{\text{I}}$ into its real and imaginary parts, we obtain
\[
\norm{ \nabla a(x) } ^2
= \sum_{k,j=1}^d (\partial_{x_j} (a_{\text{R}})_k (x) )^2 + (\partial_{x_j} (a_{\text{I}}(x))_k )^2
\]
and
\begin{align*}
& \Re \left(  \sum_{k,j=1}^d (\partial_{x_j} a_k(x))^2 \conj{\mathcal F[w_\cR]}(x,2\omega) \right) \\
& \quad = \sum_{k,j=1}^d \Re\left( (\partial_{x_j} a_k (x))^2 \right) \Re \mathcal F[w_\cR](x,2\omega) + \Im\left((\partial_{x_j} a_k (x))^2\right) \Im \mathcal F[w_\cR](x,2\omega) \\
& \quad = \sum_{k,j=1}^d \Big( (\partial_{x_j} (a_{\text{R}})_k (x))^2 - (\partial_{x_j} (a_{\text{I}})_k (x))^2 \Big) \Re \mathcal F[w_\cR](x,2\omega) \\
& \qquad + 2 \partial_{x_j} (a_{\text{R}})_k (x) \partial_{x_j} (a_{\text{I}})_k (x) \Im \mathcal F[w_\cR](x,2\omega).
\end{align*}
Combining all the terms and rearranging the summands ultimately gives the regularizer
\begin{align*}
R(\nabla a(x)) &= \sum_{k,j=1}^d \frac{1}{2} \left( \mathcal F[w_\cR](0,x) + \Re \mathcal F[w_\cR](x,2\omega)  \right) (\partial_{x_j} (a_{\text{R}})_k(x))^2 \\
& \quad + \sum_{k,j=1}^d \frac{1}{2} \left( \mathcal F[w_\cR](0,x) - \Re \mathcal F[w_\cR](x,2\omega)  \right) (\partial_{x_j} (a_{\text{I}})_k (x))^2 \\
& \quad + \sum_{k,j=1}^d \Im \mathcal F[w_\cR](x,2\omega) \partial_{x_j} (a_{\text{R}})_k(x) \partial_{x_j} (a_{\text{I}})_k (x). 
\end{align*}
The data fidelity term is transformed analogously to \autoref{thm: Horn-Schunck harmonic} resulting in
\begin{multline*}
D(a(x)) \coloneqq \int_{0}^{T} w_\cD(t,x) |G_v (t,x)|^2 \\
= \int_{0}^{T} w_\cD(t,x)
\left( \nabla I(t,x)^\tT a_{\text{R}}(x) \cos(t \omega) + \nabla I(t,x)^\tT a_{\text{I}}(x) \sin(- t \omega)  
+ \partial_tI(t,x) \right)^2 
\d t.    
\end{multline*}
In summary, the functional $E_{u, \varepsilon, \delta}(v)$ becomes
\[
E_{u, \varepsilon, \delta}(v) = \frac{1}{2}\int_{\Omega} D(a(x)) + \lambda R(\nabla a(x)) \d x + \textrm{const.}
\]
The Euler-Lagrange equations are then given by
\begin{align*}
\frac{\partial D(a(x)) }{\partial(a_{\text{R}})_k} 
& = \lambda \sum_{j=1}^d \partial_{x_j} \frac{\partial R(\nabla a(x))}{\partial (\partial_{x_j} (a_{\Re})_k (x))}, \\
\frac{\partial D(a(x)) }{\partial (a_{\text{I}})_k} 
& = \lambda \sum_{j=1}^d \partial_{x_j} \frac{\partial R(\nabla a(x))}{\partial (\partial_{x_j} (a_{\text{I}})_k(x))}
\end{align*}
for all $k\in [d]$.
The left-hand side can be transformed similarly to the proof of \autoref{thm: Horn-Schunck harmonic} with $w_\cD(t,x)$ entering the Fourier coefficients,
\begin{align*}
\frac{\partial D(a(x)) }{\partial a_{\text{R}}} 
={}&
\left( \Re \mathcal F[w_\cD \nabla I \nabla I^\tT](0,x) + \Re \mathcal F[w_\cD \nabla I \nabla I^\tT](2\omega, x) \right) a_{\text{R}}(x)
\\&
+  \Im \mathcal F[w_\cD \nabla I \nabla I^\tT](2\omega, x) a_{\text{I}}(x)
+ 2 \Re \mathcal F[w_\cD \partial_t I \nabla I](\omega, x), \\
\frac{\partial D(a(x)) }{\partial a_{\text{I}}} 
={}&
\left( \Re \mathcal F[w_\cD \nabla I \nabla I^\tT](0,x) - \Re \mathcal F[w_\cD \nabla I \nabla I^\tT](2\omega, x) \right) a_{\text{I}}(x)
\\&
+  \Im \mathcal F[w_\cD \nabla I \nabla I^\tT](2\omega, x) a_{\text{R}}(x)
+ 2 \Im \mathcal F[w_\cD \partial_t I \nabla I](\omega, x).
\end{align*}
The right-hand side reads as
\begin{align*}
\sum_{j=1}^d \partial_{x_j} \frac{\partial R(\nabla a(x))}{\partial (\partial_{x_j} (a_{\Re})_k (x))}
& = \sum_{j=1}^d \partial_{x_j} \left( \left( \mathcal F[w_\cR](0,x) + \Re \mathcal F[w_\cR](x,2\omega) \right) \partial_{x_j} (a_{\text{R}})_k (x) \right) \\
& \quad + \sum_{j=1}^d \partial_{x_j} \left( \Im \mathcal F[w_\cR](x,2\omega) \partial_{x_j} (a_{\text{I}})_k (x) \right), \\
\sum_{j=1}^d \partial_{x_j} \frac{\partial R(\nabla a_{\text{R}}(x), \nabla a_{\text{I}}(x))}{\partial (\partial_{x_j} (a_{\Im})_k (x))} 
& = \sum_{j=1}^d \partial_{x_j} \left( \left( \mathcal F[w_\cR](0,x) - \Re \mathcal F[w_\cR](x,2\omega) \right) \partial_{x_j} (a_{\text{I}})_k (x) \right) \\
& \quad + \sum_{j=1}^d \partial_{x_j} \left( \Im \mathcal F[w_\cR](x,2\omega) \partial_{x_j} (a_{\text{R}})_k (x) \right).
\end{align*}
Combining $a_{\text{R}}$ and $a_{\text{I}}$ into $a$ and noticing that $\mathcal F[w_\cD \nabla I \nabla I^\tT](0,x) \in \R^{d \times d}$ yields the desired system. 
\end{proof}

\begin{algorithm}[!b]
\caption{IRLS for Model II}\label{alg: IRLS}
\vspace{1mm}
\KwData{Images $I$, initial guess $a^0$, starting smoothing parameters $\varepsilon^0, \delta^0 > 0$, number of iterations $K \in \N$.}
\KwResult{Reconstructed density $a$.}
\For{$k=0,\ldots,K-1$}{
Compute the velocity $v^k(t,x) = \Re( a^k(x) \e^{\i t \omega})$.\\
Compute the weights $w_\cD$ and $w_\cR$ form \eqref{eq:wRD} with $v^k$, $\varepsilon^k$ and $\delta^k$.\\
Compute $a^{k+1}$, which is the minimizer of $E_{v^k,\varepsilon^k,\delta^k}$, by solving \eqref{eq: linear system weighted}.\\
Update $\varepsilon^{k+1}$, $\delta^{k+1}$ such that $\varepsilon^{k+1} \le \varepsilon^k$, $\delta^{k+1} \le \delta^k$. 
}
Set $a = a^K$.
\end{algorithm}

The reconstruction process is summarized in Algorithm \ref{alg: IRLS}.
Similarly to Model I, the size of the linear system \eqref{eq: linear system weighted} is independent of the time period $T$. However, in this case, the weights $w_\cD$ and $w_\cR$ are time-dependent and, thus, finding a solution to Model II with Algorithm \ref{alg: IRLS} is expected to be slower than to Model I. However, if the number of IRLS iterations $K$ is small, this difference is less pronounced.

We briefly discuss our update rule for the smoothing parameters $\varepsilon^k, \delta^k$ assuming that $\Omega$ has finite volume $|\Omega|$.
They are chosen as the minimum between the previous values and the scaled average of the current residual, i.e.,
\begin{align*}
\varepsilon^{k+1} & = \max\left\{ \min\left\{ \varepsilon^k, \frac{0.1}{T |\Omega| \sqrt{k+1}} \int_{\Omega} \int_{0}^{T} |G_v (t,x)| \d t \d x  \right\}, \frac{c}{\sqrt{k+1}} \right\}, \\
\delta^{k+1} & = \max\left\{ \min\left\{ \delta^{k}, \frac{0.1} {T |\Omega| \sqrt{k+1} } \int_{\Omega} \int_{0}^{T} \norm{\nabla v(t,x)}  \d t\d x \right\}, \frac{c}{\sqrt{k+1}} \right\},
\end{align*}
where $c = 10^{-8}$ is a small constant. The role of the maximum is to prevent the smoothing parameters from becoming zero. This choice does not come with theoretical convergence guarantees, but it resembles the choices in \cite[Thm. 6 and 9]{Jung.2024}. For a more comprehensive overview of possible strategies, we refer to \cite[§ 7.1]{Jung.2024}. Other than the decay strategy of the smoothing parameters, IRLS has no hyperparameters to tune, which makes its implementation and use straightforward compared to ADMM-based or primal-dual methods.

\subsection{Model III}\label{sec:L1L2}
Model III combines the approaches of the two previous sections by taking the $L_1$ data fidelity term (least absolute deviations), which is more robust to outliers,
and the $L_2$ regularizer, which promotes the smoothness of the reconstructed $a$.
In particular, we consider the energy functional
\begin{align*} 
E(v)
&\coloneqq
\int_{\Omega} \int_{0}^{T}
\left(  \abs{G_v(t,x) }
+ \lambda \norm{\nabla v(t,x)}^2
\right) \d t
\d x 
\\\nonumber
&=
\int_{\Omega}
\left(  \int_{0}^{T} \abs{G_v(t,x) }\d t
+ \frac{\lambda T}{2} \norm{\nabla a(x)}^2
\right) \d x
.
\end{align*}
Note that the data fidelity term scales linearly in $v$, while the regularizer scales quadratically. This leads to a signal- and noise-level dependent choice of the regularization parameter $\lambda$, unlike in the previous two scenarios.

The minimization can be performed with the IRLS method similarly to \autoref{sec:L1L1}.
Analogously to \eqref{eq:Eved}, we define for $\varepsilon>0$ and $u,v\colon[0,T]\times\Omega\to\R^d$ a quadratic majorant of $E$ at $u$ by
\begin{equation}\label{eq:Eved2}
E_{u, \varepsilon}(v)
\coloneqq
\int_{\Omega} \int_{0}^{T}
\left( q_\varepsilon\left( G_v (t,x), G_{u} (t,x) \right)
+ \lambda \norm{\nabla u (t,x)} ^2
\right) \d t
\d x,
\end{equation}
with $q_\varepsilon$ from \eqref{eq: majorizer q def}. Then, the functional \eqref{eq:Eved2} is iteratively minimized as in Algorithm \ref{alg: IRLS}.
The following theorem can be shown along the lines of \autoref{thm:L1L1} with $w_\cR \equiv 2$. 

\begin{theorem} \label{thm:L1L2}
Let $\varepsilon > 0$ and $u\colon[0,T]\times\Omega\to\R^d$ in $C^1$ be fixed. Consider the flow velocity $v(t,x) = \Re(a(x) \e^{\i \omega t})$ in $C^2$.  Then the amplitude $a\colon\Omega\to\C^d$ of the minimizer $v$ of  \eqref{eq:Eved2} satisfies
\begin{equation} 
\begin{aligned} \label{eq: linear system model 2}
& \mathcal F[w_\cD \nabla I \nabla I^\tT](0,x) a(x) + \mathcal F[w_\cD \nabla I \nabla I^\tT](2\omega, x) \overline{a (x)} 
+ 2 \mathcal F[w_\cD \partial_t I \nabla I](\omega, x) \\
&\quad = 2 \lambda T \Delta a(x) 
\end{aligned}
\end{equation}
for all $x\in\Omega$,
where $\mathcal F$ is the temporal Fourier transform \eqref{eq:FT} and
\begin{equation*}\label {eq:wD}
w_\cD(t, x) = \frac{1}{\max\{\varepsilon, |G_u (t,x)| \}}.
\end{equation*}
\end{theorem}

In \eqref{eq:Eved2}, the quadratic majorant $q_\varepsilon$ includes an additional factor of $\tfrac{1}{2}$ while the regularization term does not. It leads to an extra factor $2$ in the right-hand side of \eqref{eq: linear system model 2} or, equivalently, $w_\cR \equiv 2$. This is not the case for Model II, where the majorant $q_\varepsilon$ is used for both data fidelity and regularization term. Hence, the factor $\tfrac{1}{2}$ has no impact on \eqref{eq: linear system corrected}.

The solution to Model III can be found analogously to Algorithm \ref{alg: IRLS}, where $w_\cR \equiv 2$, the parameter $\delta$ does not appear and the linear system \eqref{eq: linear system model 2} is solved instead of \eqref{eq: linear system weighted}.

\section{Numerical Implementation} \label{sec:numerics_real}
In this section, we summarize implementation details for the numerical realization
of our minimization procedure. Most of these are somehow standard in optical flow computation,
but have to be adapted to our setting.
\\[2ex]
\textit{Discretization.}
We work in dimension $d = 2$ with discrete images of size $n_1 \times n_2$. The images can be seen as pointwise evaluations of the respective functions on a grid $\Gamma \coloneqq [n_1] \times [n_2]$ with notation $[n] = \{1, \ldots, n\}$. Consequently, in the implementation, all integrals $\int_{\Omega} \d x$ are replaced by sums over the grid $\sum_{x \in \Gamma}$. Similarly, the images do not form a continuum in the time variable, and thus we use $t \in [T]$ for $T \in \N$ time steps.  
Henceforth, we consider a discrete image as a 3D tensor $\zb I = [I(t, x)]_{t \in [T],\, x \in \Gamma}\in\R^{T\times n_1\times n_2}$.
We discretize the temporal Fourier transform \eqref{eq:FT} by
\begin{equation} \label{eq:DFT}
\zb F[\zb I] (\nu, x)
=
\frac1T \sum_{t=0}^{T-1} I(t, x)\, \e^{-\i t \nu}
,\qquad (\nu, x)\in[T]\times\Gamma.
\end{equation}

\textit{Derivatives.}
For the discrete approximation of the derivatives, we distinguish between $\nabla a$ and $\nabla I$. The partial derivative $\partial_1I$ is approximated via an image filtering $\zb I \star\zb h_{\mathrm{c}}$ in the spatial coordinates as follows. The convolution with the kernel
\[
\zb h_{\mathrm{c}} = \frac{1}{8}
\begin{bmatrix}
-1 & -2 & -1 \\
0 & 0 & 0 \\
1 & 2 & 1 \\   
\end{bmatrix}
\]
is defined by
\begin{equation} \label{eq:dI}
\zb D_{c,1} \zb I \coloneqq  \zb I \star\zb h_{\mathrm{c}}(t;k,j)\coloneqq \sum_{p,q=-1}^1 \zb I(t;k+p,j+q)\,\zb h_{\mathrm{c}}(p,q),\qquad (k,j)\in\Gamma,\ t\in[T],
\end{equation} 
where $\zb I$ is padded by zeros. Note that we index the elements of the matrix $\zb h_{\mathrm{c}}$ from $-1$ to $1$ in order to obtain central derivatives.
Analogously, we approximate $\partial_2I$ by $\zb I\star\zb h^\tT$.
For derivatives of $a$ and time derivative of $\zb I$, we use forward differences 
associated with the filter kernel 
$
\zb h_{\mathrm{fwd}}
=
\begin{bmatrix} 0 & -1 & 1 \end{bmatrix}^\tT,
$
i.e., with the $i$-th component $\zb a_i=[a_i(x)]_{x\in\Gamma}$, $i=1,2$, we approximate $\partial_1a_i(k,j)$ by
\begin{equation} \label{eq:da}
\zb D_1 \zb a_i(k,j) \coloneqq \zb a_i\star\zb h_{\mathrm{fwd}}(k,j)
=
\zb a_i(k+1,j) - \zb a_i(k,j)
,\qquad (k,j)\in\Gamma.
\end{equation}
In contrast to the forward differences, the usage of central differences for $\nabla a$ leads to reconstruction artifacts in our simulations. 
This is likely caused by the fact that our algorithms require the computation of $\zb D_1^\tT \zb D_1 \zb a(k,j) = \ba(k+1,j)+\ba(k-1,j)-2\ba(k,j)$, which is a good approximation of the second derivative in \eqref{eq: linear system corrected}.

\textit{Least squares implementation.}
Unlike in the classical Horn--Schunck method, which is outlined in \autoref{sec:HS-classic}, we solve the linear system \eqref{eq: linear system corrected} directly. That is, with $\zb a = \zb a_{\text{R}} +\i \zb a_{\text{I}} \in\C^{2\times n_1\times n_2}$ and the vectorization operator $\operatorname{vec}\colon\R^{2\times n_1\times n_2}\to\R^{2n_1n_2}$, the linear system \eqref{eq: linear system corrected} can be discretized in the form 
\begin{equation} \label{eq: discrete equation}
\zb C \begin{bmatrix}
\operatorname{vec}(\zb a_{\text{R}}) \\ \operatorname{vec}(\zb a_{\text{I}}) 
\end{bmatrix} = \zb b,
\end{equation}
where $\zb b$ depends only on the data $\zb I$ and the matrix $\zb C$ depends on the data $\zb I$ and the regularization parameter $\lambda$. The construction of $\zb C$ and $\zb b$ for the case $d=1$ is outlined in \autoref{sec:lsqr},
where we see that the square matrix $\zb C$ is symmetric positive semidefinite.

As the multiplication with $\zb C$ can be efficiently computed, this allows for the use of standard Matlab solvers such as the {conjugate gradient} (CG) method \texttt{pcg}, cf.\ \cite{Bar94}. Similarly, equations \eqref{eq: linear system weighted} and \eqref{eq: linear system model 2} corresponding to Model II and III, respectively, are weighted versions of \eqref{eq: linear system corrected} and can be rewritten analogously in the matrix form \eqref{eq: discrete equation}. 

\textit{Computational complexity.}
For Model I, the CG method is applied to find a solution of \eqref{eq: discrete equation}. To this end, we precompute $\zb b$ and Fourier coefficients in $\zb C$ at the expense of $\mathcal O(T n_1 n_2)$ operations, see \autoref{sec:lsqr} for details.  
The computational complexity of a single CG iteration is proportional to the cost of a matrix-vector product with the sparse matrix $\zb C$, which has a complexity of $\mathcal O(n_1 n_2)$. Consequently, for Model~I performing $K_{\textrm{CG}}$ iterations of the CG method, the total computational complexity is $\mathcal O((T + K_{\textrm{CG}}) n_1 n_2)$. 

For Models II and III, Algorithm \ref{alg: IRLS} is implemented. For each IRLS iteration, the computation of the velocity $v^k$ from given $a^k$, the weights $w_{\mathcal D}, w_{\mathcal R}$ by \eqref{eq:wRD} and the update of the smoothing parameters $\varepsilon,\delta$ require a total of $\mathcal O(T n_1 n_2)$ operations. Solving the resulting linear system has the same cost as in Model I, i.e., $\mathcal O((T + K_{\textrm{CG}}) n_1 n_2)$. Hence, running Models II and III for $K_{\textrm{irls}}$ IRLS iterations and $K_{\textrm{CG}}$ inner CG iterations requires $\mathcal O(K_{\textrm{IRLS}} (T + K_{\textrm{CG}}) n_1 n_2)$ operations. The constant in big $\mathcal O$ notation is a little smaller for Model III than for Model II because $w_\cR$ is constant and a part of the computations is not required.

\textit{Preprocessing.}
As a preprocessing step, the images are smoothed with a Gaussian kernel of standard deviation $\sigma = 0.65$. This step both filters out noise and smoothens sharp edges in the images, which have a strong impact on the reconstruction quality. 

\textit{Coarse-to-fine pyramid.}
We adopt the coarse-to-fine approach discussed in \cite[§ 4.2]{SunRotBla14}, which consists in constructing a pyramid of images $\zb I^{\ell}$ as follows. Starting with the top level $L$ as the initial image sequence $\zb I^{(L)} = \zb I$, the next level $\zb I^{(\ell)}$ is obtained by smoothing $\zb I^{(\ell+1)}$ with a Gaussian kernel and downsampling it with the factor $\eta$ such that $\zb I^{(L-1)}\in\R^{T\times\lceil\eta n_1\rceil\times\lceil\eta n_2\rceil}$ and so on. The standard deviation of the Gaussian kernel is commonly chosen as $\sigma = 1/\sqrt{2\eta}$.

The reconstruction starts with the smallest image $\zb I^{(1)}$.
In each level $\ell\in[L]$, we compute $\zb v^{(\ell)}$ as solution of the optical flow problem with the image $\zb I^{(\ell)}$, where we use as initialization in the $\ell$-th level the appropriately rescaled reconstruction $\zb v^{(\ell-1)}$. 
This yields a faster recovery as we need fewer CG iterations in each level due to the initialization. 

\textit{Median filtering.}
As a postprocessing step, a $5 \times 5$ median filter is applied to the reconstructions $\zb v^{(\ell)}$ obtained at each level of the coarse-to-fine pyramid.

\textit{Warping of derivatives.}
The coarse-to-fine approach provides intermediate solutions $\zb v^{(\ell)}$ at each level $\ell \in [L]$. As these estimated velocities are good approximations of $v$, they can be used to obtain a better approximation of the derivative $\nabla I$ in the variational model \eqref{eq:HS-general}, see 
\cite[§ 2.1]{BalBecEifFitSchSte19}.  

Denote by $\tilde{v}(t,x)$ the estimated velocity from the previous level. The first-order Taylor approximation of $\varphi$ yields
\[
\varphi(t+1,x) \approx \varphi(t,x) + \partial_t \varphi(t,x) = \varphi(t,x) + v (t,x).
\]
Rewriting \eqref{eq:dI} as 
\[
I(t,x) \approx I(t+1, x + v(t,x))  
\]
and utilizing the first-order Taylor approximation of $I(t+1, x + v (t,x))$ around the estimated velocity $x + \tilde v (t,x)$ gives
\begin{align*}
& I(t,x) 
\approx I(t+1, x + \tilde v(t,x)) - \nabla I(t+1, x  + \tilde v(t,x))^\tT (\tilde v(t,x)-v(t,x)).
\end{align*}
Hence, we have
\[
\widetilde{G_v}(t,x)
\coloneqq
\widetilde{\partial_t I(t,x)} + \widetilde{\nabla I(t,x)}^\tT v(t,x)
\approx
0  ,
\]  
where
\begin{align*}
\widetilde{\nabla I(t,x)} & \coloneqq \nabla I(t+1, x  + \tilde v(t,x)), \\
\widetilde{\partial_t I(t,x)} & \coloneqq I(t+1, x + \tilde v(t,x)) - I(t,x)  -  \widetilde{\nabla I(t,x)}^\tT \tilde v(t,x).
\end{align*}
We use $\widetilde{G_v}$ to approximate the data fidelity term.
Since $\widetilde{G_v}$ requires the evaluation of $I$ or its derivatives at points $x  + \tilde v(t,x)$, which might not be on the pixel grid $\Gamma$, we use a bicubic interpolation.

\section{Numerical Results} \label{sec:numerics}
In this section, we provide a synthetic and a real-world example
to demonstrate the performance of our three models and compare them with two optical flow algorithms created for general velocity fields, which do not need to be harmonic.
In particular, we compare the following algorithms:
\begin{itemize}
\item Model I, see \autoref{thm: Horn-Schunck harmonic},
\item Model II, see \autoref{thm:L1L1},
\item Model III, see \autoref{thm:L1L2},
\item 
the implementation of the classical Horn--Schunck method, cf. \autoref{sec:HS-classic}, available in Matlab's Image Processing Toolbox, and
\item 
the primal-dual hybrid gradient method (PDHGM) for optical flow from \cite{BalBecEifFitSchSte19,BalEifFitSchSte15} with an $L_1$ data fidelity term and an isotropic total variation regularizer akin to Model~II.
\end{itemize}
The latter two algorithms compute the velocity $v(t,\cdot)$ between \emph{consecutive pairs of images} $I(t,\cdot)$ and $I(t+1 \mod T, \cdot)$, $t \in [T]$. As the resulting $v(t,\cdot)$ may not be of the form \eqref{eq:va}, the corresponding amplitude $a$ is retrieved via the Fourier transform $ a(x) = 2 \mathcal F [v](\omega,x)$.
 
All numerical computations were performed on an 8-core Intel Core i7-10700 with 32\,GB memory running Matlab R2022a. The code for our numerical experiments will is available at \cite{code}.

\subsection{Generation of synthetic data}\label{sec: image simulation}

Before we come to the reconstruction in the next section, we describe how to generate the synthetic data:
given a velocity field $v \colon [0,T] \times \Omega \to \R$ and an initial image $I_0 \colon \Omega \to \R$,
we compute a sequence of distorted images $I(t,x)$ satisfying \eqref{eq:flow}.
Assume that $\varphi$ possesses an inverse function $\psi \colon [0,T] \times \Omega \to \Omega$ such that
\begin{equation} \label{eq:z}
    \varphi(t,\psi(t,x)) = x
    \qquad \forall (t,x)\in[0,T]\times\Omega.
\end{equation}
Substituting $x$ by $\psi(t,x)$ in \eqref{eq:bca}, we obtain
\begin{equation}\label{eq:I-consistency}
I(t,x)
=
I_0(\psi(t,x)).
\end{equation}
Thus, it remains to compute $\psi(t,x)$.
Differentiating \eqref{eq:z} for $t$, we have
\begin{align*}
0
&=
\frac{\dx}{\dx t} \varphi(t,\psi(t,x))
=
\partial_t \varphi(t,\psi(t,x))
+
\nabla \varphi(t,\psi(t,x)) \partial_t \psi(t,x)
\\
&=
v(t,x)
+
\nabla \varphi(t,\psi(t,x)) \partial_t \psi(t,x),
\end{align*}
where the Jacobian matrix is denoted by
$
\nabla \varphi(x)
=
\left[ \partial_{x_j} \varphi_i \right]_{i,j=1}^d
$.
Differentiating \eqref{eq:z} for $x$, we have
$$
\nabla x
=
\operatorname{Id}_{d\times d}
=
\nabla \varphi(t,\psi(t,x)) \nabla \psi(t,x)
$$
Hence, we obtain the initial value problem
\begin{equation}\label{eq: z generation}
\partial_t \psi(t,x)
=
- \nabla \psi(t,x) v(t,x).
\end{equation}
with $\psi(0,x) = x$. 
Each component of the last equation is equivalent to the optical flow equation \eqref{eq:of}, but with different initial values in each component.
For the numerical computation with the forward Euler method, we usually need to smoothen the gradient.

\subsection{Synthetic example}
\label{sec:numeric-sim}
In the first experiment, we test the reconstruction algorithms on a synthetic dataset. Let $n_1 = 200$, $n_2 = 206$, $T=300$ and consider the time-harmonic velocity field $v$ with frequency $\omega = 6 \pi/T$ and the complex amplitude $a\colon\Omega\to\R^2$ given by
\begin{align*}
a(x) = \begin{bmatrix}
0.8 (x_1 - \frac{n_1}{2}) \sin( \frac{\norm{x - x_0}^2}{2000}) \exp( - \frac{\norm{x - x_0}^2}{3300}) \\
10 \exp(-\frac{\norm{x-0.5 x_0}^2}{1650})
\end{bmatrix}
, \ \text{with} \ x_0 = \tfrac{1}{2} \begin{bmatrix} n_1 \\ n_2 \end{bmatrix},
\end{align*}
which is depicted in Figure \ref{fig: amplitude and images} (a) and (b). From $a$, we compute $v$ via \eqref{eq:va} and turn to the simulation of the synthetic frame sequence.
As an initial, undistorted image $I_0$ depicted in Figure \ref{fig: amplitude and images}~(c), we take the resized first frame of the ultrasound gel dataset \cite{Jor23}, which we will use in \autoref{sec:numerics_gel}.

We generate the frame sequence $\zb I = [I_0(\psi(t,x))]_{t\in[T],x\in\Gamma} \in \R^{T \times n_1 \times n_2}$ as described in \autoref{sec: image simulation}, where we evaluate $I_0$ at non-integer points or points outside the image boundary with a spline interpolation and nearest neighbor extrapolation. 


\begin{figure}
\pgfplotsset{
  colormap={parula}{
    rgb255=(53,42,135)
    rgb255=(15,92,221)
    rgb255=(18,125,216)
    rgb255=(7,156,207)
    rgb255=(21,177,180)
    rgb255=(89,189,140)
    rgb255=(165,190,107)
    rgb255=(225,185,82)
    rgb255=(252,206,46)
    rgb255=(249,251,14)}}
\centering
\subfloat[$|\zb a_1|$]{
\includegraphics[height= 82pt, trim={0 0 45 0},clip]{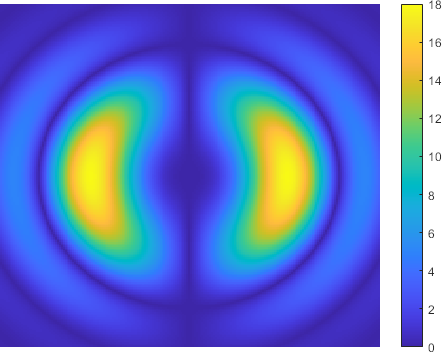}%
\begin{tikzpicture}
\pgfplotscolorbardrawstandalone[colormap name=parula,
    point meta min=0,point meta max=16, 
    colorbar style={height=82pt, width=1.4mm, xshift=0.0em, axis line style={draw opacity=0}, ytick={5,10,15}},
    yticklabel style={
      xshift=-1mm,yshift=0mm,font=\tiny
    } ]
\end{tikzpicture}\hspace{-3mm}}\
\subfloat[$|\zb a_2|$]{
\includegraphics[height= 82pt, trim={0 0 45 0},clip]{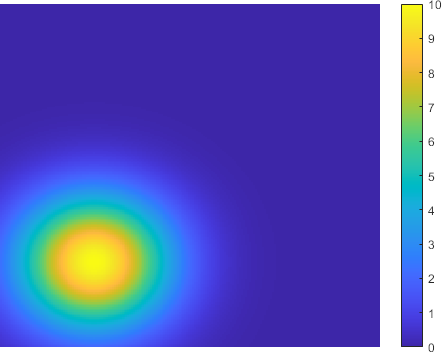}%
\begin{tikzpicture}
\pgfplotscolorbardrawstandalone[colormap name=parula,
    point meta min=0,point meta max=10, 
    colorbar style={height=82pt, width=1.4mm, xshift=0.0em, axis line style={draw opacity=0}, ytick={2,4,6,8}},
    yticklabel style={
      xshift=-1mm,yshift=0mm,font=\tiny
    } ]
\end{tikzpicture}\hspace{-1mm}}
\subfloat[$I(0,\cdot)$]{
\includegraphics[height= 82pt, trim={0 0 45 0},clip]{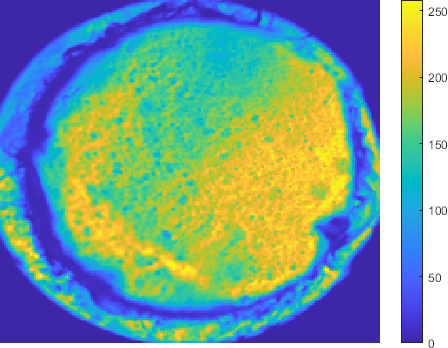}%
\begin{tikzpicture}
\pgfplotscolorbardrawstandalone[colormap name=parula,
    point meta min=0,point meta max=260, 
    colorbar style={height=82pt, width=1.4mm, xshift=0.0em, axis line style={draw opacity=0}, ytick={50,100,150,200,250}},
    yticklabel style={
      xshift=-1mm,yshift=0mm,font=\tiny
    } ]
\end{tikzpicture}\hspace{-2mm}}
\subfloat[$I_{\mathrm{noisy}}(0,\cdot)$]{
\includegraphics[height= 82pt, trim={0 0 45 0},clip]{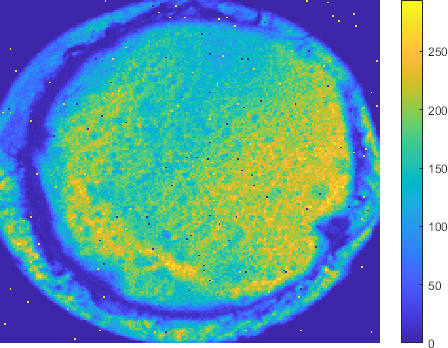}%
\begin{tikzpicture}
\pgfplotscolorbardrawstandalone[colormap name=parula,
    point meta min=0,point meta max=290, 
    colorbar style={height=82pt, width=1.4mm, xshift=0.0em, axis line style={draw opacity=0}, ytick={50,100,150,200,250}},
    yticklabel style={
      xshift=-1mm,yshift=0mm,font=\tiny
    } ]%
\end{tikzpicture}\hspace{-1mm}}
\caption{(a),\,(b): Absolute values of the two components of the ground truth amplitude~$\zb a$, (c): noiseless initial image $I_0$, (d): initial image $I_0$ corrupted by Poisson and salt-and-pepper noise.}\label{fig: amplitude and images}
\end{figure}

For the reconstruction of $\zb a \in \C^{2 \times n_1 \times n_2}$ from the generated images $\zb I$, the algorithms are executed with the following parameters. Model I uses a coarse-to-fine pyramid with 2 levels and downsampling factor $\eta = 0.8$, where for each level 50 CG iterations are performed. Models II and III use 4 levels instead. Model II performs 100 CG iterations per IRLS iteration and 5 IRLS iterations (per pyramid level), while Model III performs 25 and 4 iterations, respectively.
Matlab's Horn--Schunck does not use the coarse-to-fine approach and performs 1000 Horn--Schunck iterations per pair of images. PDHGM uses 15 pyramid levels with a downsampling ratio $\eta = 0.95$ and 50 iterations per level, which are the parameters used in \cite{BalBecEifFitSchSte19}. The regularization parameters are chosen by hand in order to receive a good result and are reported in the corresponding figures.  
We chose the mentioned numbers of iterations in a way such that more iterations yield only marginal improvements of the reconstructed images.


To measure the quality of reconstruction, the following metrics are considered. The relative error (RE) and relative image error (RIE) are given by
\[
\mathrm{RE}(\zb{\tilde a}, \zb a) \coloneqq \frac{\sum_{j=1}^2 \norm{\zb{\tilde a}_j - \zb a_j}^2}{\sum_{j=1}^2 \norm{\zb a_j }^2}
\quad \text{and} \quad 
\mathrm{RIE}(\zb{\tilde I}, \zb I) \coloneqq \frac{\sum_{t=1}^T \|\tilde I(t,\cdot) - I(t,\cdot)\|^2}{\sum_{t=1}^T \norm{I(t,\cdot)}^2},
\]
where $\|\cdot\|$ is the Euclidean norm on $\C^{n_1\times n_2}$, $\zb{\tilde a}$ is the reconstructed amplitude, and $\zb{\tilde I}$ is obtained via \eqref{eq:I-consistency} by applying to $I_0$ the time-harmonic displacement corresponding to~$\zb{\tilde a}$.
Note that the RE can be seen as an analog to the end-point error used in optical flow \cite{SunRotBla14}. We also report the structural similarity index measure (SSIM) \cite{ssim} computed for the four frames $\Re(\zb a_1)$, $\Im(\zb a_1)$, $\Re(\zb a_2)$, and $\Im(\zb a_2)$ with Matlab's \texttt{ssim}. The structural similarity index measure between the image sequences $\zb I$ and $\zb{\tilde I}$ is denoted by ISSIM. 

In our first test, we perform the reconstructions from the generated image sequence~$\zb I$ without noise. The two spatial components $a_j$ of the reconstructed vector field $a$ are shown in \autoref{fig: sim noiseless}. 
We see that all compared models yield reasonable results.
While Model~I and the reference implementation of Matlab's Horn--Schunck yield a slightly better RE, Models II and III give a better SSIM.
The error measurements for the PDHGM are in a similar magnitude, but visually the second component $a_2$ does not completely depict the structure of the ground truth.

\begin{figure}
    \centering
    \includegraphics[width=1\textwidth]{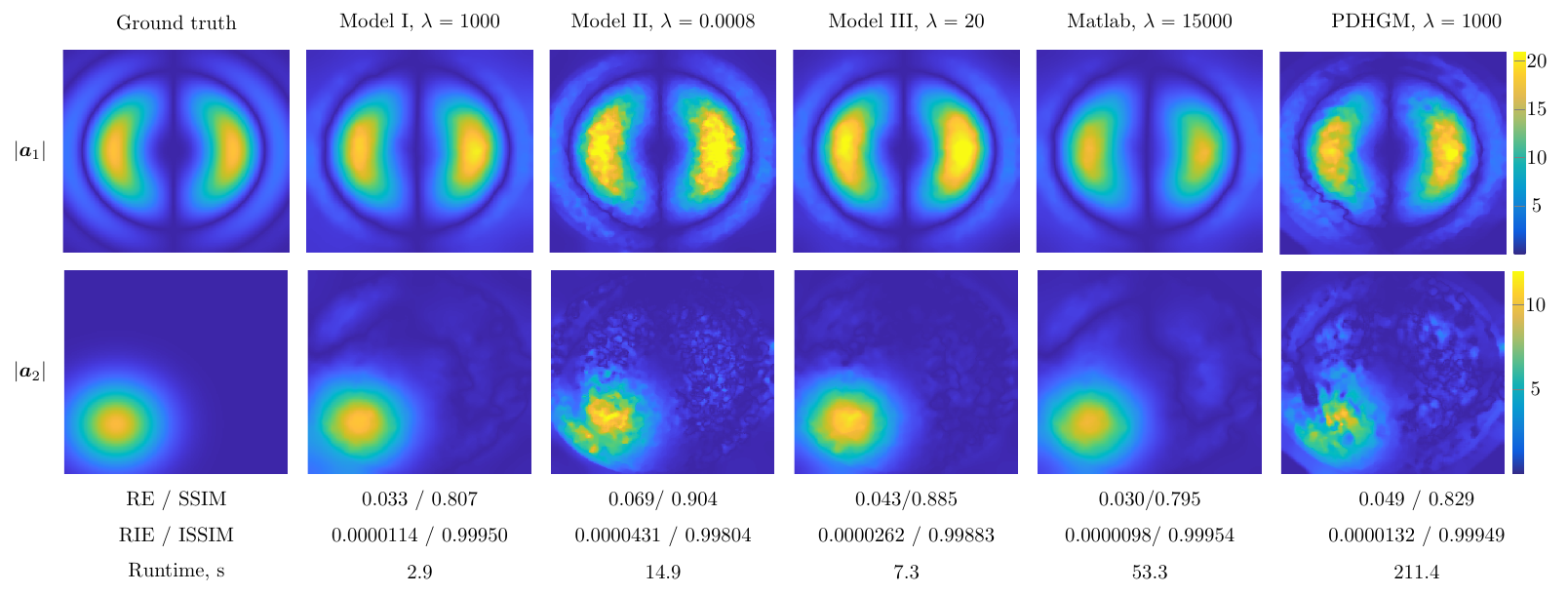}
    \caption{Absolute values of the spatial components $\zb a_j\in\C^{n_1\times n_2}$ of the reconstructed amplitude $\zb a = (\zb a_1, \zb a_2)$ from the noiseless measurements smoothed with Gaussian kernel. }
    \label{fig: sim noiseless}
\end{figure}

For our second test, the simulated images $\zb I$ are artificially corrupted by noise in two steps: first with the Poisson noise 
\[
\zb I_{\mathrm{Poisson}}(t,x) = \mathrm{Poisson}(\zb I(t,x)), \quad \text{for all} \quad x \in \Gamma, t \in [T] 
\]
and, additionally, $0.5\%$ of the pixels in $\zb I_{\mathrm{Poisson}}$ are corrupted by salt-and-pepper noise resulting in noisy images denoted by $\zb I_{\mathrm{noisy}}$. Note that the values of $\zb I$ are in the range $[0,264]$. The resulting $\mathrm{RIE}(\zb I_{\mathrm{noisy}}, \zb I)$ is $0.012$ and the first frame of $\zb I_{\mathrm{noisy}}$ is shown in Figure \ref{fig: amplitude and images} (d). The algorithmic setup for noisy data is the same as before. 

\begin{figure}[!b]
    \centering
    \includegraphics[width=1\textwidth]{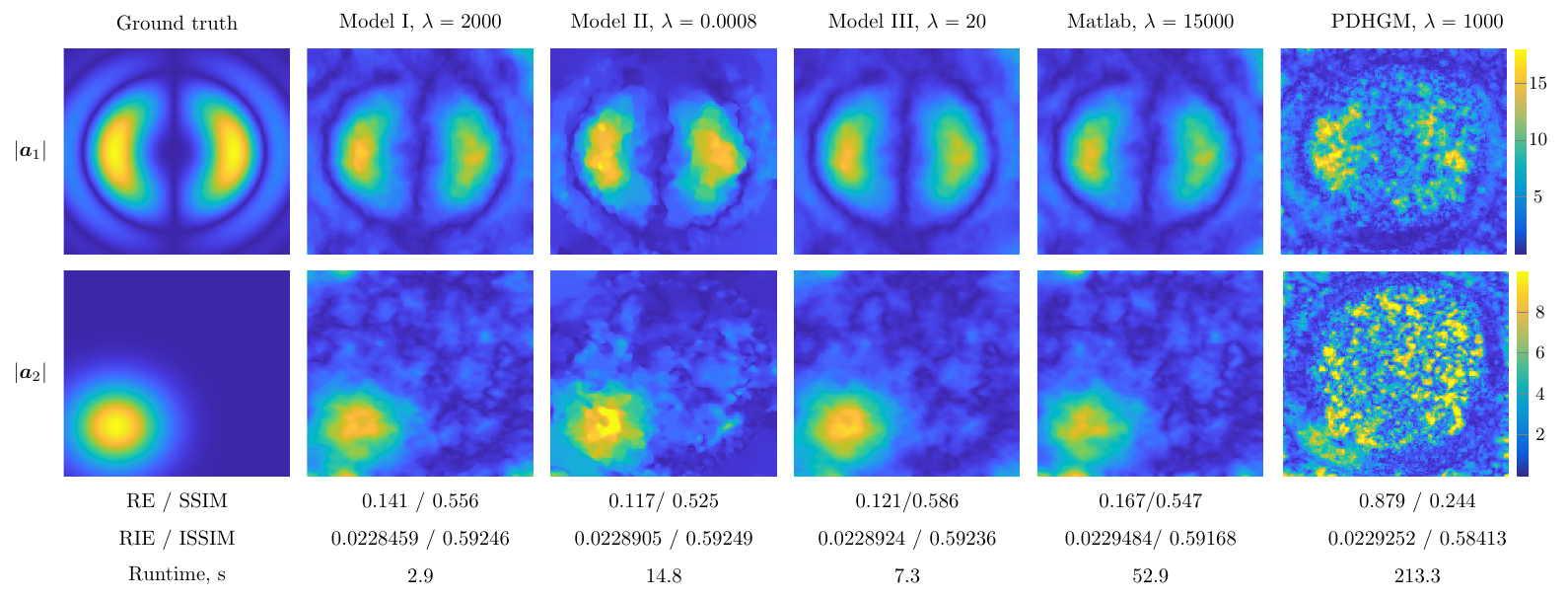}
    \caption{Absolute values of the spatial components of the reconstructed amplitude $\zb a = (\zb a_1, \zb a_2)$ from the noisy measurements $\zb I_\mathrm{n}$ smoothed with Gaussian kernel. }
    \label{fig: sim noisy}
\end{figure}

The reconstructions performed on the noisy images $\zb I_{\mathrm{noisy}}$ are shown in \autoref{fig: sim noisy}.
Due to the heavy-tailed noise, Models II and III perform best in this case, while also Model~I and Matlab's Horn--Schunck yield decent results.
The PDHGM algorithm performs considerably worse, its implementation seems to be very reliant on a good choice of the internal parameters in case of this noisy data.
Overall, the RIE must always be above a certain level due to the salt and pepper noise.
Note that we compute here the RIE between the noisy images $\zb I_{\mathrm{noisy}}$ and the images $\tilde{\zb I}$ obtained via \eqref{eq:I-consistency} with $\zb I_{\mathrm{noisy}}(0,\cdot)$ and velocity corresponding to the reconstructed amplitude $\tilde\ba$.

Considering computational times, Model~I is the fastest.
Models II and III take somewhat more time, since they have an outer IRLS loop and an inner CG loop, but the inner one requires fewer iterations than Model I since it is restarted with the previous solution.
The Matlab algorithm is much slower since it solves a $2n_1n_2\times 2n_1n_2$ linear system for each time step $t$, whereas Model~I solves only one $4n_1n_2\times 4n_1n_2$ linear system to compute $\zb a$. 

\subsection{Real-world data}
\label{sec:numerics_gel}

Now, we turn to the reconstruction from the real-world gel dataset, which has the size $(n_1,n_2)=(992,1024)$ and consists of $24$ frames stemming from $p=3$ repetitions with $T=8$ frames each. 
It was obtained by observing a phantom specimen consisting of ultrasound gel mixed with Silicon carbide scatterers. The motion was induced by targeted sound waves with a frequency of 800\,Hz perpendicular to the image plane, generated via a piezo actuator and amplifier. 
The resulting vibrations are recorded by a high-speed camera with a pixel size of $2\times2$\,µm. More details on the measurements are found in \cite{Jor23}.

For Model I, we perform 500 CG iterations without the coarse-to-fine pyramid. Models II and III incorporate the coarse-to-fine pyramid with 5 levels and downsampling factor $\eta = 0.5$. Model II performs 300 CG and 5 IRLS iterations, and Model III executes 50 CG and 6 IRLS iterations.
Matlab does not use coarse-to-fine approach and performs 1000 iterations of the Horn--Schunck algorithm per pair of images. PDHGM uses the same parameters as in \autoref{sec:numeric-sim} above. 

\begin{figure}[!t]
    \centering
    \includegraphics[width=1\textwidth]{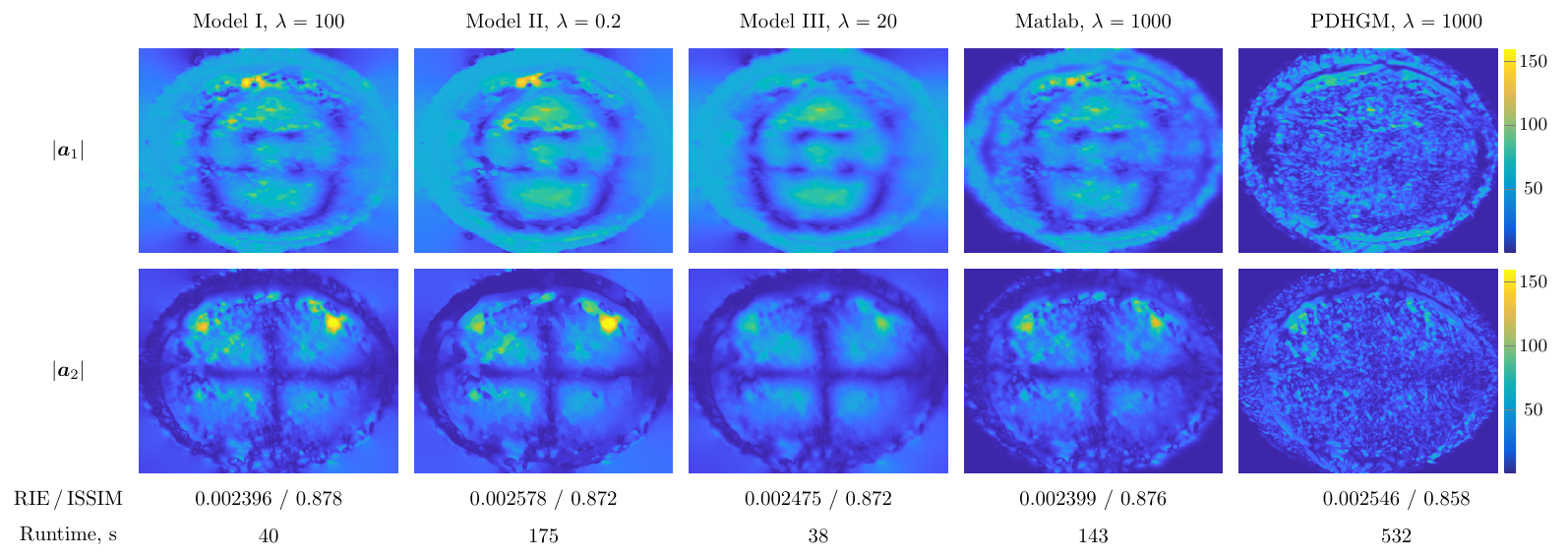}
    \caption{Absolute values of the spatial components of the reconstructed amplitude $\zb a = (\zb a_1, \zb a_2)$ from the gel dataset.}
    \label{fig: gel}
\end{figure}

The reconstructions are shown in \autoref{fig: gel}. We observe again that the accuracy of Model~I and Matlab almost matches. As only $T=24$ images are used, the speed benefit of Model I is less pronounced but still present. 
Model III is slightly faster than Model I as it uses the coarse-to-fine pyramid and performs fewer iterations on the larger images. Model II shows similar performance metrics while being slower than Model III. While the structural differences of $\zb a$ are visible, this has only a minor impact on both RIE and ISSIM.
The reconstruction with Model III is more consistent with the physically expected result for the homogeneous medium as it does not show such high peaks as the other models.
As with the noisy simulated data, the PDHGM yields an inferior result.
Furthermore, we provide videos of the velocity $\zb v$ and the images $\zb{\tilde I}$ in the supplementary material.

\section{Conclusion}
\label{sec:conclusion}

We have proposed three models for estimating the optical flow of time-harmonic oscillations.
For each of them, we have derived a fast algorithm taking advantage of the single-frequency representation of the underlying displacement velocity. The conducted numerical studies suggest that Model I is the best choice for low-noise scenarios, while Model III, which combines an $L_1$ data fidelity term with an $L_2$ smoothness term of the gradient, performs best in the presence of heavy noise. 



The next natural step in optical time-harmonic elastography consists in using our obtained velocity estimation for the reconstruction of the shear modulus of the examined material, see \cite{PapazoglouHirschBraunSack2012} or \cite[§~5]{ManducaOliphantDresnerMahowaldKruseAmromin2001}. This allows to infer information about elasticity and viscosity of observed specimen with applications to medical imaging.  


\subsection*{Acknowledgements}

This research was funded in whole or in part by the German Research Foundation (DFG): STE 571/19-1, project number 495365311, within the Austrian Science Fund (FWF) SFB 10.55776/F68 ``Tomography Across the Scales''. 
G.S. and I.S. acknowledge funding from the German Research Foundation (DFG) within the project BIOQIC (GRK2260).

For open access purposes, the authors have applied a CC BY public copyright license to any author-accepted manuscript version arising from this submission.

\printbibliography

\appendix

\

\section{Periodicity of $u$ versus $v$} \label{app:another_ex}
We look at another one-dimensional example where the displacement $u$ can be described by a single harmonic, namely
\begin{equation} \label{eq:harmonic-phi}
\varphi(t,x) = x + c(1-x^2) \sin(t)
,\qquad x\in[-1,1],\ t\ge0,
\end{equation}
with some $c>0$,
see \autoref{fig:traj} left.
Denoting again by $\psi$ the inverse of $\varphi$ as in \eqref{eq:z}, we get
$$
x=
\psi(t,y)
=\frac{1- \sqrt{ 4c^2 \sin^2(t) - 4c y \sin(t) + 1}}{2c\sin(t)}
$$
and the velocity
\begin{align*}
    v(t,y)
    &=
    \partial_t \varphi(t,\psi(t,y))
    =
    c(\psi(t,y)^2-1) \cos(t)\\
    &=
    \cos(t) \frac{\sqrt{ 4c^2 \sin^2(t) - 4cy \sin(t) + 1} + 2cy \sin(t) -1}{2c\sin^2(t)},
\end{align*}
see \autoref{fig:traj} right.
The flow velocity $v(t,\psi(t,y))$ is still $2\pi$-periodic in $t$, but it is not a single harmonic anymore.

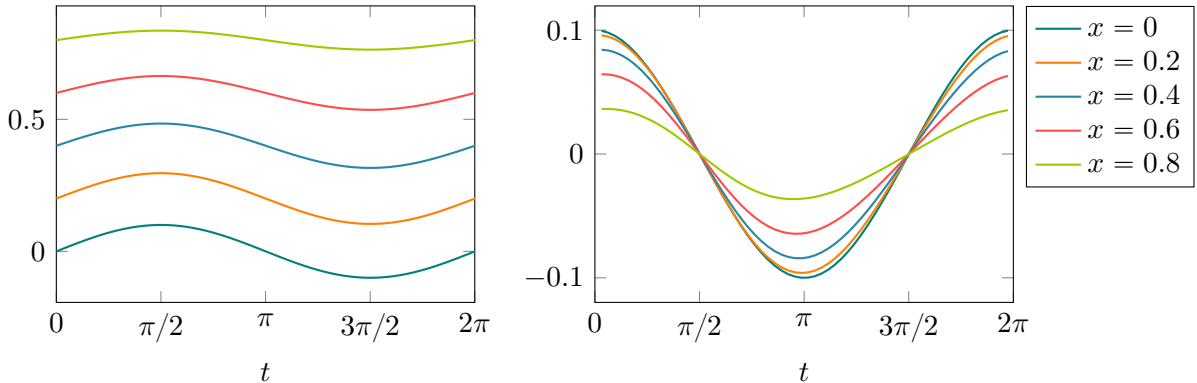
\begin{figure}[ht]
 \newcommand{\cc}{0.1}
\tikzset{font=\small}
	\begin{tikzpicture}
	\begin{axis}[xlabel=$t$, 
	width=.45\textwidth, height=0.35\textwidth, xmin=0, xmax=6.28,
	xtick={0,1.57,3.14,4.71, 6.28},
	xticklabels={0,$\pi/2$,$\pi$,$3\pi/2$,$2\pi$},
    every axis y label/.style={at={(ticklabel* cs:1.05)}, anchor=south, },
    cycle list name=exotic, 
	legend pos=south west, legend style={cells={anchor=west}}]
	\addplot+ [mark=none, thick,smooth, samples=50, domain=0.001:6.5] {\cc*sin((180*x)/pi)};
	\addplot+ [mark=none, thick,smooth, samples=50, domain=0.001:6.5] {(\cc*24*sin((180*x)/pi))/25+1/5};
	\addplot+ [mark=none, thick,smooth, samples=50, domain=0.001:6.5] {(\cc*21*sin((180*x)/pi))/25+2/5};
	\addplot+ [mark=none, thick,smooth, samples=50, domain=0.001:6.5] {(\cc*16*sin((180*x)/pi))/25+3/5};
	\addplot+ [mark=none, thick,smooth, samples=50, domain=0.001:6.5] {(\cc*9*sin((180*x)/pi))/25+4/5};
	\end{axis}
	\end{tikzpicture}\hfill
	\begin{tikzpicture}
	\begin{axis}[xlabel=$t$, 
	width=.45\textwidth, height=0.35\textwidth, xmin=0, xmax=6.28,
	xtick={0,1.57,3.14,4.71,6.28},
	xticklabels={0,$\pi/2$,$\pi$,$3\pi/2$,$2\pi$},
    scaled ticks=false,
    yticklabel style={/pgf/number format/fixed,},
    every axis y label/.style={at={(ticklabel* cs:1.05)}, anchor=south, },
    cycle list name=exotic, 
	legend pos=outer north east, legend style={cells={anchor=west}}]
	\addplot+ [mark=none, thick] table[x index=0, y index=1] {figures/ex1.dat};
	\addplot+ [mark=none, thick] table[x index=0, y index=2] {figures/ex1.dat};
	\addplot+ [mark=none, thick] table[x index=0, y index=3] {figures/ex1.dat};
	\addplot+ [mark=none, thick] table[x index=0, y index=4] {figures/ex1.dat};
	\addplot+ [mark=none, thick] table[x index=0, y index=5] {figures/ex1.dat};
	\legend{$x=0$,$x=0.2$,$x=0.4$,$x=0.6$,$x=0.8$};
	\end{axis}
	\end{tikzpicture}
    \caption{Trajectories $\varphi(t,x)$ (left) and flow velocity $v(t,x)$ (right) corresponding to \eqref{eq:harmonic-phi} for different values of $x$ and fixed $c=0.1$ and $\omega=1$.\label{fig:traj}}
\end{figure}

\section{The classical Horn--Schunck algorithm}
\label{sec:HS-classic}
The classical Horn--Schunck algorithm as proposed in \cite{HornSchunck1981} can be straightforwardly adopted to solve \eqref{eq: linear system corrected} in the following way. 
We again split up real and imaginary part, as the complex equation contains both $a$ and its conjugate $\overline{a}$. Then, in matrix form, \eqref{eq: linear system corrected} reads as 
\begin{equation*}
\lambda T
\begin{pmatrix}
    \Delta a_{\text{R}}(x)\\
    \Delta a_{\text{I}}(x)
\end{pmatrix}
=
C(x)
\begin{pmatrix}
    a_{\text{R}}(x)\\
    a_{\text{I}}(x)
\end{pmatrix}
+ 2 
\begin{pmatrix}
\Re \mathcal F[\partial_t I \nabla_x I^\tT](\omega, x) \\
\Im \mathcal F[\partial_t I \nabla_x I^\tT](\omega, x) 
\end{pmatrix}
\end{equation*}
with the coefficient matrix
\begin{equation}\label{eq: C coef matrix}
C(x)
=
\begin{pmatrix}
\Re \mathcal F[\nabla_x I \nabla_x I^\tT](2\omega, x) 
&
\Im \mathcal F[\nabla_x I \nabla_x I^\tT](2\omega, x)
\\
\Im \mathcal F[\nabla_x I \nabla_x I^\tT](2\omega, x) 
&
- \Re \mathcal F[\nabla_x I \nabla_x I^\tT](2\omega, x) 
\end{pmatrix}
+
\mathrm I_2\otimes \mathcal F[\nabla_x I \nabla_x I^\tT](x, 0),
\end{equation}
where $\otimes$ is the Kronecker product and $\mathrm I_2$ is the $2\times2$ identity matrix.
By rewriting the system as
\begin{equation*}
\lambda T
\begin{pmatrix}
    (\Delta - 1) a_{\text{R}}(x)\\
    (\Delta - 1) a_{\text{I}}(x)
\end{pmatrix}
=
\left[C(x) + \lambda T\, \mathrm I_{2d}\right]
\begin{pmatrix}
    a_{\text{R}}(x)\\
    a_{\text{I}}(x)
\end{pmatrix}
+ 2 
\begin{pmatrix}
\Re \mathcal F[\partial_t I \nabla_x I^\tT](\omega, x) \\
\Im \mathcal F[\partial_t I \nabla_x I^\tT](\omega, x) 
\end{pmatrix},
\end{equation*}
the right-hand side depends only on $a_{\text{R}}$ and $a_{\text{I}}$ evaluated at a single point $x$. If the left-hand side were independent of the values of $a$ at points other than $x$, we could significantly simplify the computational effort and solve a separate linear system for each $x$. 
However, the Laplacian requires the values of $a_{\text{R}}$ and $a_{\text{I}}$ in a neighborhood of $x$. To avoid a large linear system, one solves \eqref{eq: linear system corrected} with the following iterative procedure. Starting with initial $a_{\text{R}}^0$ and $a_{\text{I}}^0$, we construct for $k \ge 0$ new iterates $a_{\text{R}}^{k+1}$ and $a_{\text{I}}^{k+1}$ by solving for each $x$ the $2d\times 2d$ linear system 
\begin{equation*}
\lambda T
\begin{pmatrix}
    (\Delta - 1) a_{\text{R}}^k(x)\\
    (\Delta - 1) a_{\text{I}}^k(x)
\end{pmatrix}
=
\left[C(x) + \lambda T\, \mathrm I_{2d} \right]
\begin{pmatrix}
    a_{\text{R}}^{k+1}(x)\\
    a_{\text{R}}^{k+1}(x)
\end{pmatrix}
+ 2 
\begin{pmatrix}
\Re \mathcal F[\partial_t I \nabla_x I^\tT](\omega, x) \\
\Im \mathcal F[\partial_t I \nabla_x I^\tT](\omega, x) 
\end{pmatrix}.
\end{equation*}
When discretizing $a$ on a uniform grid, $(\Delta-1)a(x)$ is implemented as the mean value over the grid points adjacent to $x$.

\section{Construction of the linear system for Model I}
\label{sec:lsqr}

In this section, we consider the discretization of \eqref{eq: linear system corrected} in a one-dimensional setting. Let $d=1$, $n_1 = n$ and assume that $T=2p\pi/\omega \in \N$ for some $p\in\N$. 
While it is possible to directly replace the operators in \eqref{eq: linear system corrected} with their discrete analogues, we first go a step back and start with the minimization of the energy functional $E$ for Model I. Its empirical counterpart $E_\mathrm{emp}\colon\C^n\to\R$ for observed images $\zb I \in \R^{T \times n}$ is given by
\begin{align*}
E_\mathrm{emp}(\zb a) 
& =
\sum_{x \in \Gamma} \sum_{t = 0}^{T-1}
\left( 
\left( \zb D_{c} \zb I(t,x)^\top \zb a_{\text{R}}(x) \cos(-t \omega) + \zb D_c \zb I(t,x)^\top \zb a_{\text{I}}(x) \sin(- t \omega)  
+ \zb D_t \zb I(t,x) \right)^2 \right. \\
& \left. \qquad \qquad + \lambda | \zb D \zb a_{\text{R}}(x) \cos(-t \omega) + \zb D \zb a_{\text{I}}(x) \sin(- t \omega) |^2
\right),
\end{align*}
where $\Gamma = [n]$ is the pixel grid, $\zb D_c$ denotes the one-dimensional discrete central derivatives as in \eqref{eq:dI} and  
$\zb D, \zb D_t$ are the forward difference operators \eqref{eq:da} in space and time, respectively. As $d = 1$, we leave out the extra index for the simplicity of notation.

Minimizing $E_\mathrm{emp}$ can be viewed as a least squares problem. Introducing the matrices  
\[
\zb y \coloneqq - \begin{bmatrix}
\zb D_t \zb I(\cdot,0) \\
\vdots  \\
\zb D_t \zb I(\cdot,T-1) \\
\end{bmatrix},
\quad
\zb R \coloneqq 
\begin{bmatrix}
\zb D & 0 \\
0 & \zb D     
\end{bmatrix},
\quad
\zb M \coloneqq 
\begin{bmatrix}
\zb M_{0,1} & \zb M_{0,2} \\
\vdots  & \vdots \\
\zb M_{T-1,1} & \zb M_{T-1,2} 
\end{bmatrix},
\]
with blocks 
\[
\zb M_{t,1} \coloneqq \diag(\zb D_c \zb I(t,\cdot)) \cos( - 2 \pi t p / T),
\quad \text{and} \quad 
\zb M_{t,2} \coloneqq \diag(\zb D_c \zb I(t,\cdot)) \sin(- 2 \pi t p / T),
\]
and simplifying the regularizer as in \eqref{eq: regularizer simplified} transforms the energy functional $E_\mathrm{emp}$ into
\begin{equation*}\label{eq: linear system Horn Schunck}
E_\mathrm{emp} (\zb a) 
= 
\left\| \zb M \begin{bmatrix}
\zb a_{\text{R}} \\ \zb a_{\text{I}} 
\end{bmatrix} - \zb y \right\|_2^2 + \frac{\lambda T}{2} \left\| \zb R
\begin{bmatrix}
\zb a_{\text{R}} \\ \zb a_{\text{I}} 
\end{bmatrix} 
\right\|_2^2
=
\left\| \begin{bmatrix} \zb M\\ \sqrt{\frac{\lambda T}{2}} \zb R \end{bmatrix} \begin{bmatrix}
\zb a_{\text{R}} \\ \zb a_{\text{I}} 
\end{bmatrix} 
- \begin{bmatrix} \zb y \\ 0 \end{bmatrix} \right\|_2^2.
\end{equation*} 
Its normal equation, which is the discrete analogue of the Euler--Lagrange equations that lead us to \eqref{eq: linear system corrected}, is given by
\begin{equation} \label{eq:HS-disc}
\zb C \begin{bmatrix} \zb a_{\text{R}} \\ \zb a_{\text{I}} \end{bmatrix} 
\coloneqq \left(\zb M^\top \zb M + \frac{\lambda T}{2} \zb R^\top \zb R \right)  \begin{bmatrix} \zb a_{\text{R}} \\ \zb a_{\text{I}} \end{bmatrix} 
= \zb M^\top \zb y
\eqqcolon \zb b,
\end{equation}
which is precisely \eqref{eq: discrete equation}. Moreover, $\zb C$ is square, symmetric and positive semidefinite. 

It remains to argue that \eqref{eq:HS-disc} corresponds to the direct discretization of \eqref{eq: linear system corrected}.
For this, we observe that the matrix $\zb C$ and the vector $\zb b$ admit the block structure
\[
\zb C = \begin{bmatrix}
\zb C_{1,1} & \zb C_{1,2} \\
\zb C_{2,1} & \zb C_{2,2}
\end{bmatrix},
\quad 
\zb b = \begin{bmatrix}
\zb b_1 \\
\zb b_2 \\
\end{bmatrix},
\]
with
\begin{align*}
\zb C_{1,1} & = \sum_{t = 0}^{T-1} \diag(\zb D_c \zb I(t,\cdot) \circ \zb D_c \zb I(t,)) \cos^2 ( - t \omega) + \frac{\lambda T}{2} \zb D^\top \zb D \\
& = \sum_{t = 0}^{T-1} \diag(\zb D_c \zb  I(t,\cdot) \circ \zb D_c \zb I(t,\cdot)) \left[ \frac{1}{2 } + \frac{\cos (- 2 \pi t p / T)}{2} \right] + \frac{\lambda T}{2} \zb D^\top \zb D \\
& = \diag \left( \frac{1}{2} \zb F [\zb D_c \zb I \circ \zb D_c \zb I] (0,\cdot) + \frac{1}{2} \Re \zb F [\zb D_c \zb I \circ \zb D_c \zb I] (2 p, \cdot) \right) + \frac{\lambda T}{2} \zb D^\top \zb D, 
\end{align*}
and, similarly,
\begin{align*}
\zb C_{1,2} & = \zb C_{2,1} = \diag \left( \frac{1}{2} \Im \zb F [\zb D_c \zb I\circ \zb D_c \zb I] (2 p, \cdot) \right), \\
\zb C_{2,2} & = \diag \left( \frac{1}{2} \zb F [\zb D_c \zb I \circ \zb D_c \zb I] (0, \cdot) - \frac{1}{2} \Re \zb F [\zb D_c \zb I\circ \zb D_c \zb I] (2p, \cdot) \right) + \frac{\lambda T}{2} \zb D^\top \zb D, \\
\zb b_1 & = \Re\left( \zb F [\zb D_c \zb I \circ \zb D_t \zb I] (p, \cdot) \right),
\qquad \zb b_2  = \Im\left( \zb F [\zb D_c \zb I \circ \zb D_t \zb I] (p, \cdot) \right),
\end{align*}
where $\circ$ denotes the entrywise product and $\zb F$ is the discrete Fourier transform \eqref{eq:DFT}.
Since $\zb D^\top \zb D$ is a discrete approximation of the negative Laplacian $-\Delta$,
we see that \eqref{eq:HS-disc} is indeed a discretization of \eqref{eq: linear system corrected}.

\end{document}